\documentclass[a4paper,12pt]{article}

\usepackage[english]{babel}
\usepackage[T1]{fontenc}

\usepackage{tikz}
\tikzstyle{vertex}=[circle, draw, inner sep=2pt, minimum size=6pt]

\usepackage{mathrsfs}

\usepackage[a4paper,top=3cm,bottom=2cm,left=3cm,right=3cm,marginparwidth=1.75cm]{geometry}

\usepackage{comment}
\usepackage{amsfonts}
\usepackage{fullpage}
\usepackage{latexsym}
\usepackage{amsfonts}
\usepackage{blkarray}
\usepackage{graphicx}
\usepackage{float}
\usepackage{enumerate}
\usepackage{amsthm,amsmath,amssymb}
\usepackage{verbatim}
\usepackage{enumitem}
\usepackage{blkarray}
\usepackage{mathtools}
\usepackage{algorithm}
\usepackage[noend]{algpseudocode}

\providecommand{\keywords}[1]{
  \small	
  \textbf{\textit{Keywords---}} #1
}

\newtheorem{lemma}{Lemma}[section]
\newtheorem{corollary}{Corollary}[section]
\newtheorem{proposition}{Proposition}[section]

\newtheorem{example}{Example}[section]

\newcommand{\nc}{\newcommand}
\newcommand{\Yildirim}{Y{\i}ld{\i}r{\i}m}

\nc{\cA}{{\cal A}}
\nc{\cB}{{\cal B}}
\nc{\cC}{{\cal C}}
\nc{\cD}{{\cal D}}
\nc{\cE}{{\cal E}}
\nc{\cG}{{\cal G}}
\nc{\cF}{{\cal F}}
\nc{\cH}{{\cal H}}
\nc{\cI}{{\cal I}}
\nc{\cK}{{\cal K}}
\nc{\cL}{{\cal L}}
\nc{\cM}{{\cal M}}
\nc{\cN}{{\cal N}}
\nc{\cO}{{\cal O}}
\nc{\cP}{{\cal P}}
\nc{\cQ}{{\cal Q}}
\nc{\cR}{{\cal R}}
\nc{\cS}{{\cal S}}
\nc{\cT}{{\cal T}}
\nc{\cV}{{\cal V}}
\nc{\tx}{{\tilde x}}
\nc{\la}{{\langle}}
\nc{\ra}{{\rangle}}
\nc{\ts}{\textsuperscript}
\def\R{\mathbb{R}}

\title{Polyhedral Properties of RLT Relaxations of 
Nonconvex Quadratic Programs and Their Implications on Exact Relaxations}
\author{Yuzhou Qiu\thanks{School of Mathematics, Peter Guthrie Tait Road, The University of Edinburgh, Edinburgh, EH9 3FD, United Kingdom. E-mail: \tt{y.qiu-16@sms.ed.ac.uk}} \and E. Alper \Yildirim\thanks{School of Mathematics, Peter Guthrie Tait Road, The University of Edinburgh, Edinburgh, EH9 3FD, United Kingdom. ORCID ID: 0000-0003-4141-3189 E-mail: \tt{E.A.Yildirim@ed.ac.uk}}}
\date{24 March 2023}

\begin{document}

\maketitle
\begin{abstract}
  We study linear programming relaxations of nonconvex quadratic programs given by the reformulation-linearization technique (RLT), referred to as RLT relaxations. We investigate the relations between the polyhedral properties of the feasible regions of a quadratic program and its RLT relaxation. We establish various connections between recession directions, boundedness, and vertices of the two feasible regions. Using these properties, we present necessary and sufficient exactness conditions for RLT relaxations. We then give a thorough discussion of how our results can be converted into simple algorithmic procedures to construct instances of quadratic programs with exact, inexact, or unbounded RLT relaxations. 
\end{abstract}



\section{Introduction} \label{intro}

A quadratic program involves minimizing a (possibly nonconvex) quadratic function over a polyhedron:
\[
\textrm{(QP)} \quad \ell^* = \min\limits_{x \in \R^n} \left\{q(x): x \in F\right\},
\]
where $q: \R^n \to \R$ and $F \subseteq \R^n$ are given by
\begin{eqnarray}
q(x) & = & \textstyle\frac{1}{2} x^T Q x + c^T x, \label{def_qx} \\ 
F & = & \left\{x \in \R^n: G^T x \leq g, \quad H^T x = h \right\}. \label{def_F}
\end{eqnarray}
Here, $Q \in \R^{n \times n}$, $c \in \R^n$, $G \in \R^{n \times m}$, $H \in \R^{n \times p}$, $g \in \R^m$, and $h \in \R^p$ constitute the parameters and $x \in \R^n$ denotes the decision variable. Without loss of generality, we assume that $Q$ is a symmetric matrix. We denote the optimal value of (QP) by $\ell^* \in \R \cup \{-\infty\} \cup \{+\infty\}$, with the usual conventions for infeasible and unbounded problems.

Quadratic programs constitute an important class of problems in global optimization and they arise in a wide variety of applications, ranging from support vector machines and portfolio optimization to various combinatorial optimization problems such as the maximum stable set problem and the MAX-CUT problem. We refer the reader to~\cite{FuriniTBFGGLLMM19} and the references therein. In addition, quadratic programs also appear as subproblems in sequential quadratic programming algorithms for solving more general classes of nonlinear optimization problems (see, e.g.,~\cite{NoceWrig06}).   

It is well-known that quadratic programs, in general, are NP-hard (see, e.g.,~\cite{Sahni74,Pardalos1991QuadraticPW}). As such, global optimization algorithms for quadratic programs are generally based on spatial branch-and-bound methods. Convex relaxations play the crucial role of generating lower bounds in this framework. 

In this paper, we focus on a linear programming relaxation of (QP) arising from the reformulation-and-linearization technique (RLT), henceforth referred to as the RLT relaxation~\cite{Sherali1999}. The RLT relaxation consists of two stages. The reformulation stage consists of generating quadratic constraints that are implied by the linear constraints in $F$. Such quadratic constraints are obtained either by multiplying two linear inequality constraints or by multiplying a linear equality constraint by a variable. In the linearization stage, the resulting implied quadratic constraints are linearized by introducing a new variable for each quadratic term. Finally, the substitution of quadratic terms in the objective function of (QP) with the new variables gives rise to the RLT relaxation, whose optimal value, denoted by $\ell_R^*$, yields a lower bound on $\ell^*$. We say that the RLT relaxation is \emph{exact} if $\ell_R^* = \ell^*$.

We investigate the relations between the polyhedral properties of the feasible region $F$ of (QP) given by \eqref{def_F} and that of its RLT relaxation, denoted by $\cF$. In particular, we focus on the relations between recession directions, boundedness, and vertices of the two feasible regions. As a byproduct of our analysis, we identify a necessary and sufficient condition in order for an RLT relaxation to be exact. We also discuss how our results can be used to construct an instance of (QP) that admits an exact, inexact, or unbounded RLT relaxation. Our contributions are as follows:

\begin{enumerate}
    \item We show that various properties of $F$ such as boundedness and existence of vertices directly translate to $\cF$.

    \item We present simple procedures for constructing recession directions and vertices of $\cF$ from their counterparts in $F$.

    \item For a certain subclass of quadratic programs, we obtain a complete description of the set of all vertices. By observing that every quadratic program can be equivalently reformulated in this form, we discuss the implications of this observation on RLT relaxations of general quadratic programs. 

    \item We identify a necessary and sufficient condition for instances of (QP) that admit an exact RLT relaxation. 
    
    \item By using the aforementioned exactness characterization together with the optimality conditions of the RLT relaxation, we present simple algorithmic procedures to construct an instance of (QP) with an unbounded, inexact, or exact RLT relaxation.
\end{enumerate}

We consider the polyhedron $F$ in the general form given by \eqref{def_F} as opposed to a more convenient form such as the standard form for the following reasons. First, many classes of problems such as quadratic programs with box constraints and standard quadratic programs have an associated natural formulation, and converting it to another form generally requires the introduction of additional variables and/or constraints. Such a conversion increases the dimension of the corresponding RLT relaxation, which, in turn, increases the computational cost of solving the relaxation. Second, such a conversion may change the polyhedral structure of the feasible region of (QP). For instance, $F$ may not have any vertices but any nonempty polyhedron in standard form necessarily has at least one vertex. On the other hand, our goal in this paper is to identify the relations between the polyhedral properties of the original feasible region $F$ and those of $\cF$. As such, we adopt the general description given by \eqref{def_F}.

This paper is organized as follows. We review the literature in Section~\ref{lit_rev} and define our notation in Section~\ref{notation}. We present basic results about polyhedra in Section~\ref{Sec2}. We introduce the RLT relaxation and discuss its polyhedral properties in Section~\ref{Sec3}.  Section~\ref{Sec4} is devoted to the discussion of duality and optimality conditions of the RLT relaxation. We introduce the convex underestimators induced by the RLT relaxation and present necessary and sufficient conditions for an exact RLT relaxation in Section~\ref{Sec5}. We discuss how our results can be used to efficiently construct instances of (QP) with exact, inexact, or unbounded RLT relaxations in Section~\ref{Sec6}. 
Finally, Section~\ref{Sec7} concludes the paper. 

\subsection{Literature Review} \label{lit_rev}

In this section, we briefly review the relevant literature. 

The ideas that led to the RLT relaxation were developed by several authors in a series of papers. To the best of our knowledge, the terminology first appears in~\cite{SheraliA92}, where the authors develop a branch-and-bound method based on RLT relaxations for solving bilinear quadratic programs, i.e., instances of (QP) for which all the diagonal entries of $Q$ are equal to zero. In~\cite{sherali1995reformulation}, this approach is extended to general quadratic programs and several properties of the RLT relaxation are established. The RLT relaxation has been extended to more general classes of discrete and nonconvex optimization problems~(see, e.g., \cite{Sherali1999}).

RLT relaxations of quadratic programs can be further strengthened by adding a set of convex quadratic constraints~\cite{sherali1995reformulation} or by adding semidefinite constraints~\cite{anstreicher2009semidefinite}, referred to as the SDP-RLT relaxation. The latter relaxation usually provides much tighter bounds than the RLT relaxation at the expense of significantly higher computational effort. Furthermore, a continuum of linear programming relaxations between the RLT relaxation and the SDP-RLT relaxation can be obtained by viewing the semidefinite constraint as an infinite number of linear constraints and adding these linear cuts in a cutting plane framework~\cite{SheraliF02}. Alternatively, by using another representation as an infinite number of second-order conic constraints, one can obtain a sequence of second-order conic relaxations that are provably tighter than their linear programming counterparts~\cite{KimK03}. In terms of computational cost, second-order conic relaxations roughly lie between cheaper linear programming relaxations and more expensive semidefinite programming relaxations. Alternative convex relaxations can be obtained by relying on the observation that every quadratic program can be equivalently formulated as an instance of a copositive optimization problem~\cite{Burer09}, which is a convex but NP-hard problem. Nevertheless, the copositive cone can be approximated by various sequences of tractable convex cones, each of which gives rise to relaxation hierarchies that are exact in the limit. We refer the reader to~\cite{yildirim2020alternative} for a unified treatment of a rather large family of convex relaxations arising from the copositive formulation, and to~\cite{BaoST11,anstreicher2012convex,Bomze15} for comparisons of various convex relaxations. 

Despite the fact that there exist many convex relaxations that are provably at least as tight as the RLT relaxation, the latter is used extensively in global optimization algorithms (see, e.g.,~\cite{Sherali1999, audet2000branch,TawarmalaniS05}) due to the fact that state-of-the-art linear programming solvers can usually scale very well with the size of the problem. Furthermore, they are generally much more numerically stable than second-order conic programming and semidefinite programming solvers that are required for solving tighter relaxations. As such, RLT relaxations play a central role in global solution algorithms for nonconvex optimization problems, which motivates our focus on their polyhedral properties.

Recently, the authors of this paper studied RLT and SDP-RLT relaxations of quadratic programs with box constraints, which is a special case of (QP)~\cite{QY23}. They presented algebraic descriptions of instances that admit exact RLT relaxations as well as those that admit exact SDP-RLT relaxations. Using these descriptions, they proposed simple algorithmic procedures for constructing instances with exact or inexact RLT and SDP-RLT relaxations. Some of our results in this paper can be viewed as extensions of the corresponding results in~\cite{QY23} to general quadratic programs. In contrast with~\cite{QY23}, where the focus is on descriptions of instances with exact and inexact relaxations for the specific class of quadratic programs with box constraints, our main focus in this paper is on the relations between the polyhedral properties of general quadratic programs and their RLT relaxations.

\subsection{Notation} \label{notation}

We use $\R^n$, $\R^n_+$, $\R^{m \times n}$, $\cS^n$, and $\cN^n$ to denote the $n$-dimensional Euclidean space, the nonnegative orthant, the set of $m \times n$ real matrices, the space of $n \times n$ real symmetric matrices, and the cone of componentwise nonnegative $n \times n$ real symmetric matrices, respectively. We use 0 to denote the real number 0, the vector of all zeroes, as well as the matrix of all zeroes, which should always be clear from the context. We denote by $e \in \R^n$ the vector of all ones. All inequalities on vectors or matrices are componentwise. The rank of a matrix $A \in \R^{m \times n}$ is denoted by $\textrm{rank}(A)$. We use $\textrm{conv}(\cdot)$, $\textrm{cone}(\cdot)$, and $\textrm{span}(\cdot)$ to denote the convex hull, conic hull, and the collection of all linear combinations of a set, respectively. For index sets $\mathbf{J} \subseteq \{1,\ldots,m\}$, $\mathbf{K} \subseteq \{1,\ldots,n\}$, $x \in \R^n$, and $B \in \R^{m \times n}$, we denote by $x_\mathbf{K} \in \R^{|\mathbf{K}|}$ the subvector of $x$ restricted to the indices in $\mathbf{K}$ and by $B_{\mathbf{J}\mathbf{K}} \in \R^{|\mathbf{J}|\times|\mathbf{K}|}$ the submatrix of $B$ whose rows and columns are indexed by $\mathbf{J}$ and $\mathbf{K}$, respectively, where $|\cdot|$ denotes the cardinality of a finite set. We use $x_j$ and $Q_{ij}$ for singleton index sets. For any $U \in \R^{m \times n}$ and $V \in \R^{m \times n}$, the trace inner product is denoted by 
\[
\langle U, V \rangle = \textrm{trace}(U^T V) = \sum\limits_{i=1}^m \sum\limits_{j = 1}^n U_{ij} V_{ij}.
\]

\section{Preliminaries} \label{Sec2}

In this section, we review basic facts about polyhedra. We refer the reader to~\cite{Schrijver1999} for proofs and further results.

Let $F \subseteq \R^n$ be a nonempty polyhedron given by \eqref{def_F}. 
The recession cone of $F$, denoted by $F_\infty \subseteq \R^n$, is given by 
\begin{equation} \label{def_Dinf}
F_\infty = \left\{d \in \R^n: G^T d \leq 0, \quad H^T d = 0\right\}.     
\end{equation}

Note that $F_\infty$ is a polyhedral cone. By the Minkowski-Weyl Theorem, it is finitely generated, i.e., there exists $d^j \in \R^n,~j = 1,\ldots,t$ such that 
\begin{equation} \label{def_Dinf_alt}
F_\infty = \textrm{cone}\left(\left\{d^1,\ldots,d^t\right\}\right) = \left\{\sum\limits_{j=1}^t \lambda_j d^j: \lambda_j \geq 0,~j = 1,\ldots,t\right\}.     
\end{equation}

Recall that a hyperplane $\{x \in \R^n: a^T x = \alpha\}$, where $a \in \R^n$ and $\alpha \in \R$, is a \emph{supporting hyperplane} of $F$ if $\alpha = \min\{a^T x: x \in F\}$. A set $F_0 \subseteq F$ is a \emph{face} of $F$ if $F_0$ is given by the intersection of $F$ with a supporting hyperplane. In particular, $F_0 \subseteq F$ is a face of $F$ if and only if there exists a submatrix $G^0 \in \R^{n \times m_0}$ of $G$, where $m_0 \leq m$, such that 
\begin{equation} \label{face_def}
 F_0 = \{x \in F: (G^{0})^T x = g^{0}\},   
\end{equation}
where $g^0 \in \R^{m_0}$ denotes the corresponding subvector of $g$. In particular, $F_0$ is a \emph{minimal face} of $F$ if and only if it is an affine subspace, i.e., if and only if there exists a submatrix $G^0 \in \R^{n \times m_0}$ of $G$, where $m_0 \leq m$, and a corresponding subvector $g^0 \in \R^{m_0}$ of $g$ such that
\begin{equation} \label{min_face_def}
 F_0 = \{x \in \R^n: (G^0)^T x = g^0, \quad H^T x = h\}.   
\end{equation}

Let 
\begin{equation} \label{rank-const-mat}
\rho= \textrm{rank}\left(\begin{bmatrix} G & H \end{bmatrix} \right) \leq n.
\end{equation}

The dimension of each minimal face $F_0 \subseteq F$ is equal to $n - \rho$. Note that every polyhedron has a finite number of minimal faces. In particular, if $\rho = n$, then each minimal face $F_0 \subseteq F$ consists of a single point called a \emph{vertex}. This gives rise to the following useful characterizations of vertices. 

\begin{lemma} \label{vertex_chars}
Let $F \subseteq \R^n$ be a nonempty polyhedron given by \eqref{def_F} and let $\hat x \in F$. Then, the following statements are equivalent:
\begin{enumerate}
    \item[(i)] $\hat x$ is a vertex of $F$.
    \item[(ii)] $\hat x - \hat d \in F$ and $\hat x + \hat d \in F$ if and only if $\hat d = 0$.
    \item[(iii)] There exists a partition of $G = \begin{bmatrix} G^0 &  G^1\end{bmatrix}$ and a corresponding partition of $g^T = \begin{bmatrix} (g^0)^T & (g^1)^T \end{bmatrix}$ such that $(G^0)^T \hat x = g^0$, $(G^1)^T \hat x < g^1$, and the matrix $\begin{bmatrix} G^0 & H \end{bmatrix}$ has full row rank.
    \item[(iv)] There exists $a \in \R^n$ such that $\hat x$ is the unique optimal solution of $\min\{a^T x: x \in F\}$.
    \item[(v)] $0 \in \R^n$ is a vertex of $F_\infty$.
\end{enumerate}   
\end{lemma}

Next, we collect several results concerning the recession cone $F_\infty$ given by \eqref{def_Dinf}.

\begin{lemma} \label{rec_cone_results}
Let $F \subseteq \R^n$ be a nonempty polyhedron given by \eqref{def_F}. Then, the following statements are equivalent:
\begin{enumerate}
    \item[(i)] $F$ has no vertices.
    \item[(ii)] $F$ contains a line.
    \item[(iii)] $\rho < n$, where $\rho$ is defined as in \eqref{rank-const-mat}.
    \item[(iv)] There exists $\hat d \in \R^n \backslash \{0\}$ such that $\hat d \in F_\infty$ and $-\hat d \in F_\infty$ (i.e., $F_\infty$ contains a line).
    \item[(v)] $F_\infty$ has no vertices.
\end{enumerate}
\end{lemma}

Recall that $F$ is a polytope if it is bounded. In this case, $F_\infty = \{0\}$. We next state a useful characterization of polytopes. 

\begin{lemma} \label{bounded-F}
Let $F \subseteq \R^n$ be a nonempty polyhedron given by \eqref{def_F}. Then, $F$ is bounded if and only if, for every $z \in \R^n$, there exists $(u,v) \in \R^m \times \R^p$ such that 
\begin{equation} \label{span-F}
G u + H w = z, \quad u \geq 0.
\end{equation}
\end{lemma}
\begin{proof}
Since $F$ is nonempty, the boundedness of $F$ is equivalent to 
\[
F_\infty = \left\{d \in \R^n: G^T d \leq 0, \quad H^T d = 0 \right\} = \{0\}. 
\]
Therefore, $F$ is bounded if and only if, for every $z \in \R^n$, the optimal value of the linear programming problem
\[
\max\{z^T d: G^T d \leq 0, \quad H^T d = 0\}
\]
is equal to zero. The assertion follows from linear programming duality.
\end{proof}

Finally, we close this section with a useful decomposition result.

\begin{lemma} \label{decomp_result}
Let $F \subseteq \R^n$ be a nonempty polyhedron given by \eqref{def_F}, and let $F_i \subseteq F,~i = 1,\ldots,s$ denote the set of minimal faces of $F$. Then, 
\begin{equation} \label{decomp_polyhedra}
F = \textrm{conv}\left(\left\{v^1,\ldots,v^s\right\}\right) + F_\infty,    
\end{equation}
where $v^i \in F_i,~i = 1,\ldots,s$, and $F_\infty$ is given by \eqref{def_Dinf}.
\end{lemma}

\section{Polyhedral Properties of RLT Relaxations} \label{Sec3}

In this section, given an instance of (QP), we introduce the corresponding RLT (reformulation-linearization technique) relaxation. We then focus on the relations between the polyhedral properties of the feasible region of (QP) and that of its RLT relaxation. We establish several connections between recession directions, boundedness, and vertices of the two feasible regions. For a specific class of quadratic programs, we give a complete characterization of the set of vertices of the feasible region of the RLT relaxation. We finally discuss the implications of this observation on RLT relaxations of general quadratic programs.

\subsection{RLT Relaxations}

Recall that an instance of (QP) is completely specified by the objective function $q(x)$ and the feasible region $F$ given by \eqref{def_qx} and \eqref{def_F}, respectively. The RLT relaxation of (QP) is obtained by generating quadratic constraints implied by linear constraints. Such quadratic constraints are obtained by multiplying each pair of linear inequality constraints and by multiplying each linear equality constraint by a variable. Note that it is not necessary to add the quadratic constraints obtained from multiplying each pair of equality constraints since they are already implied by the aforementioned procedure (see, e.g.,~\cite[Remark 1]{sherali1995reformulation}). The resulting quadratic constraints and the objective function are then linearized by substituting each quadratic term $x_i x_j$ by a new variable $X_{ij},~i = 1,\ldots,n; j = 1,\ldots,n$. 

For a given instance of (QP), the RLT relaxation of (QP) is therefore given by
\[
\textrm{(RLT)} \quad \ell^*_R = \min\limits_{x \in \R^n, X \in \cS^n} \left\{\textstyle\frac{1}{2}\langle Q, X \rangle + c^T x : (x,X) \in \cF\right\},
\]
where 
\begin{equation} \label{def_cF}
\cF = \left\{(x,X) \in \R^n \times \cS^n: \begin{array}{rcl}     
G^T x & \leq & g\\
H^T x & = & h \\
H^T X & = & h x^T \\
G^T X G - G^T x g^T - g x^T G + g g^T & \geq & 0
\end{array}
\right\}.
\end{equation}

Note that (RLT) is a linear programming relaxation of (QP) since, for each $\hat x \in F$, we have $(\hat x, \hat x \hat x^T) \in \cF$ with the same objective function value. Therefore, 
\begin{equation} \label{rlt-lb}
    \ell^*_R \leq \ell^*.
\end{equation}

\subsection{Recession Cones and Boundedness}

In this section, we present several relations between the recession cones associated with the polyhedral feasible regions $F$ and $\cF$ of (QP) and (RLT), respectively. We also discuss the boundedness relation between $F$ and $\cF$.

Recall that the recession cone of $F$, denoted by $F_\infty$, is given by \eqref{def_Dinf}. Similarly, we use $\cF_\infty$ to denote the recession cone of $\cF$, which is given by 
\begin{equation} \label{def_cDinf}
\cF_\infty = \left\{(d,D) \in \R^n \times \cS^n: \begin{array}{rcl} G^T d & \leq & 0 \\ H^T d & = & 0 \\ H^T D - h d^T & = & 0 \\ 
G^T D G - G^T d g^T - g d^T G & \geq & 0 \end{array} \right\}.     
\end{equation}

Note that 
\begin{equation} \label{rec_dir_imp}
(\hat d, \hat D) \in \cF_\infty \Longrightarrow \hat d \in F_\infty.
\end{equation}

Our next result gives a recipe for constructing recession directions of $\cF$ from recession directions of $F$.

\begin{lemma} \label{rel-rec-cones}
Let $F \subseteq \R^n$ be a nonempty polyhedron given by \eqref{def_F} and let $P  = \begin{bmatrix} d^1 & \cdots & d^t \end{bmatrix} \in \R^{n \times t}$, where $d^1,\ldots,d^t$ are defined as in \eqref{def_Dinf_alt}. Then, for each $\hat d \in F_\infty$, each $\hat x \in F$, and each $\hat K \in \cN^t$, we have $(\hat d, \hat D) \in \cF_\infty$, where $\hat D = 
\hat x \hat d^T + \hat d \hat x^T + P \hat K P^T \in \cS^n$.
\end{lemma}
\begin{proof}
Since $\hat d \in F_\infty$, we have $G^T \hat d \leq 0$ and $H^T \hat d = 0$ by \eqref{def_Dinf}. Furthermore, we have
\[
H^T \hat D - h \hat d^T = H^T \hat x \hat d^T + H^T \hat d \hat x^T + H^T P \hat K P^T - h \hat d^T = h \hat d^T - h \hat d^T = 0,
\]
where we used $H^T \hat x = h$, $H^T \hat d = 0$, and $H^T P = 0$. Finally,
\begin{footnotesize}
\begin{eqnarray*}
G^T \hat D G - G^T \hat d g^T - g \hat d^T G & = &  (G^T \hat x) (G^T \hat d)^T + (G^T \hat d) (G^T \hat x)^T + G^T P \hat K P^T G -  (G^T \hat d) g^T - g (G^T \hat d)^T \\
 & = & (G^T \hat x - g) (G^T \hat d)^T + (G^T \hat d) (G^T \hat x - g)^T + (G^T P) \hat K (G^T P)^T\\
 & \geq & 0,
\end{eqnarray*}
\end{footnotesize}
where the last inequality follows from $G^T \hat x \leq g$, $G^T \hat d \leq 0$, $G^T P \leq 0$, and $\hat K \geq 0$. Therefore, $(\hat d, \hat D) \in \cF_\infty$ by \eqref{def_cDinf}.
\end{proof}

Next, we discuss the boundedness relation between $F$ and $\cF$. 

\begin{lemma} \label{bounded_imp}
$F$ is nonempty and bounded if and only if $\cF$ is nonempty and bounded.
\end{lemma}
\begin{proof}
Suppose that $F$ is nonempty and bounded. Then, $F_\infty = \{0\}$. Clearly, $\cF$ is nonempty since, for each $\hat x \in F$, we have $(\hat x, \hat x \hat x^T) \in \cF$. Let $(\hat d, \hat D) \in \cF_\infty$. By \eqref{rec_dir_imp}, we obtain $\hat d \in F_\infty$, which implies that $\hat d = 0$. By \eqref{def_cDinf}, 
\[
H^T \hat D = 0, \quad G^T \hat D G \geq 0.
\]
By Lemma~\ref{bounded-F}, for every $z^1 \in \R^n$ and $z^2 \in \R^n$, there exist $u^1 \in \R^m_+$, $w^1 \in \R^p$, $u^2 \in \R^m_+$, and $w^2 \in \R^p$ such that $G u^1 + H w^1 = z^1$ and $G u^2 + H w^2 = z^2$. Therefore, for every $z^1 \in \R^n$ and $z^2 \in \R^n$,
\[
(z^1)^T \hat D z^2 = (G u^1 + H w^1)^T \hat D (G u^2 + H w^2) = (u^1)^T G^T \hat D G u^2 \geq 0,
\]
where we used $H^T \hat D = 0$, $G^T \hat D G \geq 0$, $u^1 \geq 0$, and $u^2 \geq 0$. Since the inequality above  holds for every $z^1 \in \R^n$ and $z^2 \in \R^n$, we obtain $\hat D = 0$. Therefore, $(\hat d, \hat D) = (0,0)$, which implies that $\cF_\infty = \{(0,0)\}$. It follows that $\cF$ is bounded. 

Conversely, if $\cF$ is nonempty and bounded, then $F$ is nonempty and bounded since $F$ is the projection of $\cF$ onto the $x$-space.
\end{proof}

\subsection{Vertices} \label{vertices}

In this section, we focus on the relations between the vertices of $F$ and those of $\cF$. First, we consider the case in which $F$ has no vertices.

\begin{lemma} \label{novertex_imp}
Suppose that $F$ is nonempty. If $F$ has no vertices, then $\cF$ has no vertices.
\end{lemma}
\begin{proof}
By Lemma~\ref{rec_cone_results}~(iv), 
there exists a nonzero $\hat d \in \R^n \backslash \{0\}$ such that $\hat d \in F_\infty$ and $-\hat d \in F_\infty$. Let $\hat x \in F$ and define $\hat D = \hat x \hat d^T + \hat d \hat x^T$. By Lemma~\ref{rel-rec-cones}, we obtain $(\hat d, \hat D) \in \cF_\infty$ and $-(\hat d, \hat D) \in \cF_\infty$, which implies that $\cF$ contains a line. By Lemma~\ref{rec_cone_results}, $\cF$ has no vertices. 
\end{proof}

Before we present the relations between the set of vertices of $F$ and that of $\cF$, we state a useful technical lemma that will be helpful in the remainder of this section.

\begin{lemma} \label{matrix-equation}
Let $A \in \R^{n \times k}$ and $Z \in \cS^k$. Then, the system $A^T W A = Z$ has a solution $W \in \cS^n$ if and only if the range space of $Z$ is contained in the range space of $A^T$. Furthermore, if $A$ has full row rank, then the solution is unique.
\end{lemma}
\begin{proof}
If $A^T W A = Z$ has a solution $W \in \cS^n$, then, for any $y \in \R^k$, we have $Z y = A^T (W A y)$, which implies that the range space of $Z$ is contained in the range space of $A^T$. 

Conversely, let $Z = \sum\limits_{j=1}^\kappa \lambda_j z^j (z^j)^T$ denote the eigenvalue decomposition of $Z$, where $\kappa \leq k$ denotes the rank of $Z$ and $z^j \in \R^k,~j = 1,\ldots,\kappa$. By the hypothesis, the range space of $Z$, given by $\textrm{span}\{z^1,\ldots,z^\kappa\}$, is contained in the range space of $A^T$. Therefore, for each $j = 1,\ldots,\kappa$, there exists $u^j \in \R^n$ such that $z^j = A^T u^j$. It follows that $Z = A^T U \Lambda U^T A$, where $U = \begin{bmatrix} u^1 & \cdots & u^\kappa \end{bmatrix} \in \R^{n \times \kappa}$ and $\Lambda \in \cS^{\kappa}$ is a diagonal matrix whose entries are given by $\lambda_1,\ldots,\lambda_{\kappa}$. Therefore, $W = U \Lambda U^T$ is a solution of $A^T W A = Z$. 

If $A$ has full row rank, then the uniqueness of the solution $W \in \cS^n$ follows from the observation that the matrix $U$ is uniquely determined.
\end{proof}

We are now in a position to present the first relation between the set of vertices of $F$ and that of $\cF$.

\begin{proposition} \label{vertex1}
Suppose that $F$ is nonempty. Let $\hat x \in F$ and $\hat X = \hat x {\hat x}^T \in \cS^n$. Then, $(\hat x, \hat X)$ is a vertex of $\cF$ if and only if $\hat x$ is a vertex of $F$.   
\end{proposition}
\begin{proof}
For each $\hat x \in F$, we clearly have $(\hat x, \hat X) \in \cF$, where $\hat X = \hat x {\hat x}^T \in \cS^n$. 

First, suppose that $\hat x \in \R^n$ is a vertex of $F$. Let $G^0 \in \R^{n \times m_0}$ and $G^1 \in \R^{n \times m_1}$ denote the submatrices of $G$, where $m_0 + m_1 = m$, and let $g^0 \in \R^{m_0}$ and $g^1 \in \R^{m_1}$ denote the corresponding subvectors of $g$ such that 
\begin{equation} \label{def_G0_G1}
(G^0)^T \hat x = g^0, \quad (G^1)^T \hat x < g^1.
\end{equation} 
First, we identify the set of active constraints of (RLT) at $(\hat x, \hat X)$:
\begin{footnotesize}
\begin{eqnarray*}
(G^0)^T \hat x & = & g^0\\
(G^1)^T \hat x & < & g^1\\
H^T \hat x & = & h \\
H^T \hat x \hat x^T & = & h \hat x^T\\
\begin{bmatrix} (G^0)^T \\ (G^1)^T \end{bmatrix} \hat x \hat x^T  \begin{bmatrix} (G^0)^T \\ (G^1)^T \end{bmatrix}^T - \begin{bmatrix} (G^0)^T \\ (G^1)^T \end{bmatrix} \hat x \begin{bmatrix} g^0 \\ g^1 \end{bmatrix}^T - \begin{bmatrix} g^0 \\ g^1 \end{bmatrix} \hat x^T \begin{bmatrix} (G^0)^T \\ (G^1)^T \end{bmatrix}^T + \begin{bmatrix} g^0 \\ g^1 \end{bmatrix} \begin{bmatrix} g^0 \\ g^1 \end{bmatrix}^T & = & \begin{bmatrix} 0 & 0 \\ 0 & r^1 (r^1)^T \end{bmatrix},
\end{eqnarray*}
\end{footnotesize}
where $r^1 = g^1 - (G^1)^T \hat x > 0$. Therefore, $r^1 (r^1)^T$ is componentwise strictly positive. 

By Lemma~\ref{vertex_chars}~(ii), it suffices to show that the system 
\begin{eqnarray*}
(G^0)^T \hat d & = & 0\\
H^T \hat d & = & 0 \\
H^T \hat D & = & h \hat d^T\\
(G^0)^T \hat D G^0 - (G^0)^T \hat d (g^0)^T - g^0 \hat d^T G^0 & = & 0\\
(G^0)^T \hat D G^1 - (G^0)^T \hat d (g^1)^T - g^0 \hat d^T G^1 & = & 0, 
\end{eqnarray*}
where $(\hat d, \hat D) \in \R^n \times \cS^n$, 
has a unique solution $(\hat d, \hat D) = (0,0)$. 

Since $\hat x$ is a vertex of $F$, the matrix $\begin{bmatrix} G^0 & H \end{bmatrix}$ has full row rank by Lemma~\ref{vertex_chars}~(iii). Therefore, we obtain $\hat d = 0$ from the first two equations. Substituting $\hat d = 0$ into the third and fourth equations, we obtain
\[
\begin{bmatrix} (G^0)^T \\ H^T \end{bmatrix} \hat D \begin{bmatrix} (G^0)^T \\ H^T \end{bmatrix}^T = \begin{bmatrix} 0 & 0 \\ 0 & 0 \end{bmatrix},
\]
where we used $H^T \hat D = 0$ by the third equation. 
By Lemma~\ref{matrix-equation}, we obtain $\hat D = 0$, which implies that $(\hat x, \hat X)$ is a vertex of $\cF$.

Conversely, suppose that $\hat x$ is not a vertex of $F$. By Lemma~\ref{vertex_chars}~(ii), there exists a nonzero $\hat d \in \R^n$ such that each of $\hat d$ and $-\hat d$ is a feasible direction at $\hat x \in F$. Using the same partition as in \eqref{def_G0_G1}, we obtain
\[
(G^0)^T \hat d = 0, \quad H^T \hat d = 0.
\]
Let $\hat D = \hat d \hat x^T + \hat x \hat d^T \in \cS^n$. 
We claim that each of $(\hat d, \hat D)$ and $-(\hat d, \hat D)$ is a feasible direction at $(\hat x, \hat X)$. Indeed, 
\begin{eqnarray*}
H^T (\hat d \hat x^T + \hat x \hat d^T) & = & h \hat d^T\\
(G^0)^T (\hat d \hat x^T + \hat x \hat d^T) G^0 - (G^0)^T \hat d (g^0)^T - g^0 \hat d^T G^0 & = & 0\\
(G^0)^T (\hat d \hat x^T + \hat x \hat d^T) G^1 - (G^0)^T \hat d (g^1)^T - g^0 \hat d^T G^1 & = & 0     
\end{eqnarray*}
Furthermore, since $(G^1)^T \hat x < g^1$ and $(G^1)^T \hat X G^1 - (G^1)^T \hat x (g^1)^T - g^1 \hat x^T G^1 + g^1 (g^1)^T = r^1 (r^1)^T > 0$, where $r^1 = g^1 - (G^1)^T \hat x > 0$, it follows that there exists a real number $\epsilon > 0$ such that $(\hat x, \hat X) + \epsilon (\hat d, \hat D) \in \cF$ and $(\hat x, \hat X) - \epsilon (\hat d, \hat D) \in \cF$, which implies that $(\hat x, \hat X)$ is not a vertex of $\cF$ by Lemma~\ref{vertex_chars}~(ii).
\end{proof}

By Proposition~\ref{vertex1}, for each vertex $\hat x \in F$, there is a corresponding vertex $(\hat x, \hat X) \in \cF$, where $\hat X = \hat x \hat x^T$. We therefore obtain the following result.

\begin{corollary} \label{vertex_existence}
Suppose that $F$ is nonempty. The set of vertices of $F$ is nonempty if and only if the set of vertices of $\cF$ is nonempty.    
\end{corollary}
\begin{proof}
The result immediately follows from Lemma~\ref{novertex_imp} and Proposition~\ref{vertex1}.    
\end{proof}

Next, we identify another connection between the set of vertices of $\cF$ and the set of vertices of $F$.

\begin{proposition} \label{midpoint-vertex}
Let $v^1 \in F$ and $v^2 \in F$ be two vertices such that $v^1 \neq v^2$. Let $\hat x = \textstyle\frac{1}{2}(v^1 + v^2)$ and $\hat X = \textstyle\frac{1}{2}\left(v^1 (v^2)^T + v^2 (v^1)^T\right)$. Then, $(\hat x, \hat X)$ is a vertex of $\cF$.
\end{proposition}
\begin{proof}
Let $v^1 \in F$ and $v^2 \in F$ be two vertices. Let $\hat x = \textstyle\frac{1}{2}(v^1 + v^2)$ and $\hat X = \textstyle\frac{1}{2}\left(v^1 (v^2)^T + v^2 (v^1)^T\right)$. First, we verify that $(\hat x, \hat X) \in \cF$. Clearly, we have
\begin{eqnarray*}
G^T \hat x = G^T \left(\textstyle\frac{1}{2} (v^1 + v^2) \right) & \leq & g \\
H^T \hat x = H^T \left(\textstyle\frac{1}{2} (v^1 + v^2) \right) & = &h \\
H^T \hat X = H^T \left(\textstyle\frac{1}{2} \left(v^1 (v^2)^T + v^2 (v^1)^T\right) \right) = h \left(\textstyle\frac{1}{2} (v^1 + v^2) \right)^T & = & h \hat x^T.
\end{eqnarray*}
Let us define 
\begin{equation} \label{residuals}
r^{(1)} = g - G^T v^1 \geq 0, \quad r^{(2)} = g - G^T v^2 \geq 0.
\end{equation}
Then, we obtain 
\begin{equation} \label{prod_ineq}
G^T \hat X G - G^T \hat x g^T - g \hat x^T G + g g^T = \textstyle\frac{1}{2} \left(r^{(1)} (r^{(2)})^T + r^{(2)} (r^{(1)})^T \right) \geq 0,    
\end{equation}
where we used \eqref{residuals}. Therefore, $(\hat x, \hat X) \in \cF$. 

Next, we show that $(\hat x, \hat X)$ is a vertex of $\cF$. 
We define the following submatrices of $G$ and the corresponding subvectors of $g$:
\begin{eqnarray*}
(G^0)^T \hat x \, \, = \, \, (G^0)^T v^1 & = & (G^0)^T v^2 = g^0\\
(G^1)^T v^1 = g^1 & , & (G^1)^T v^2 < g^1\\
(G^2)^T v^1 < g^2 & , & (G^2)^T v^2 = g^2\\
(G^3)^T v^1 < g^3 & , & (G^3)^T v^2 < g^3.
\end{eqnarray*}
We remark that $G^1$ and $G^2$ are nonempty submatrices of $G$ since $v^1 \neq v^2$. By \eqref{residuals} and \eqref{prod_ineq}, 
we can identify the set of active constraints of (RLT) at $(\hat x, \hat X)$:
\begin{eqnarray*}
(G^0)^T \hat x & = & g^0\\
(G^j)^T \hat x & < & g^j, \quad j = 1,2,3\\
H^T \hat x & = & h \\
H^T \hat X & = & h \hat x^T\\
G^T \hat X G - G^T \hat x g^T - g \hat x^T G + g g^T & = & \begin{bmatrix} 0 & 0 & 0 & 0 \\
0 & 0 & + & + \\ 0 & + & 0 & + \\ 0 & + & + & + \end{bmatrix},
\end{eqnarray*}
where, in the last equation, we assume without loss of generality that $G = \begin{bmatrix} G^0 & G^1 & G^2 & G^3 \end{bmatrix}$ and $g$ is partitioned accordingly, and $+$ denotes a submatrix with strictly positive entries.

Therefore, by Lemma~\ref{vertex_chars}~(ii), it suffices to show that the system 
\begin{eqnarray*}
(G^0)^T \hat d & = & 0\\
H^T \hat d & = & 0 \\
H^T \hat D & = & h \hat d^T\\
(G^i)^T \hat D G^j - (G^i)^T \hat d (g^j)^T - g^i \hat d^T G^j & = & 0, \quad (i,j) \in \{(0,0),(0,1),(0,2),(0,3),(1,1),(2,2)\},
\end{eqnarray*}
where $(\hat d, \hat D) \in \R^n \times \cS^n$, has a unique solution $(\hat d, \hat D) = (0,0)$. 

Therefore, $(\hat d, \hat D)$ should simultaneously solve the following two systems:
\begin{eqnarray*}
\begin{bmatrix} (G^0)^T \\ (G^1)^T \\ H^T \end{bmatrix} \hat D \begin{bmatrix} (G^0)^T \\ (G^1)^T \\ H^T \end{bmatrix}^T & = &  \begin{bmatrix} 0 & g^0 \hat d^T G^1 & 0 \\ (G^1)^T \hat d (g^0)^T & (G^1)^T \hat d (g^1)^T + g^1 \hat d^T G^1 & (G^1)^T \hat d h^T \\ 0 & h \hat d^T G^1 & 0 \end{bmatrix}, \\
\begin{bmatrix} (G^0)^T \\ (G^2)^T \\ H^T \end{bmatrix} \hat D \begin{bmatrix} (G^0)^T \\ (G^2)^T \\ H^T \end{bmatrix}^T & = &  \begin{bmatrix} 0 & g^0 \hat d^T G^2 & 0 \\ (G^2)^T \hat d (g^0)^T & (G^2)^T \hat d (g^2)^T + g^2 \hat d^T G^2 & (G^2)^T \hat d h^T \\ 0 & h \hat d^T G^2 & 0 \end{bmatrix}.
\end{eqnarray*}
Substituting $(G^0)^T v^1 = g^0$ and $H^T v^1 = h$ into the first equation, and $(G^0)^T v^2 = g^0$ and $H^T v^2 = h$ into the second one, we obtain
\begin{eqnarray*}
\begin{bmatrix} (G^0)^T \\ (G^1)^T \\ H^T \end{bmatrix} \hat D \begin{bmatrix} (G^0)^T \\ (G^1)^T \\ H^T \end{bmatrix}^T & = & 
\begin{bmatrix} (G^0)^T \\ (G^1)^T \\ H^T \end{bmatrix} \left(v^1 \hat d^T + \hat d (v^1)^T\right) \begin{bmatrix} (G^0)^T \\ (G^1)^T \\ H^T \end{bmatrix}^T \\
\begin{bmatrix} (G^0)^T \\ (G^2)^T \\ H^T \end{bmatrix} \hat D \begin{bmatrix} (G^0)^T \\ (G^2)^T \\ H^T \end{bmatrix}^T & = & 
\begin{bmatrix} (G^0)^T \\ (G^2)^T \\ H^T \end{bmatrix} \left(v^2 \hat d^T + \hat d (v^2)^T\right) \begin{bmatrix} (G^0)^T \\ (G^2)^T \\ H^T \end{bmatrix}^T .
\end{eqnarray*}
Since $v^1$ and $v^2$ are vertices of $F$, Lemma~\ref{vertex_chars}~(iii) implies that each of the two matrices $\begin{bmatrix} G^0 & G^1 & H 
\end{bmatrix} $ and $\begin{bmatrix} G^0 & G^2 & H 
\end{bmatrix}$ has full row rank. By Lemma~\ref{matrix-equation}, we obtain
\[
\hat D = v^1 \hat d^T + \hat d (v^1)^T =  v^2 \hat d^T + \hat d (v^2)^T,
\]
which implies that $(v^1 - v^2) \hat d^T + \hat d (v^1 - v^2)^T = 0$. Since $v^1 \neq v^2$, we obtain $\hat d = 0$. Substituting this into the matrix equations above, we obtain $\hat D = 0$ by Lemma~\ref{matrix-equation}, which proves the assertion.
\end{proof}

Propositions~\ref{vertex1} and \ref{midpoint-vertex} identify two sets of vertices of $\cF$ under the assumption that $F$ contains at least one vertex. An interesting question is whether every vertex of $\cF$ belongs to one of these two sets. The following example illustrates that this is not necessarily true.

\begin{example} \label{Example1}
Let $n = 2$ and 
\[
F = \left\{x \in \R^2: 0 \leq x_j \leq 1, \quad j = 1,2\right\}, 
\]
i.e., we have $p = 0$, $m = 4$, $G = \begin{bmatrix} I & - I \end{bmatrix}$, and $g^T = \begin{bmatrix} e^T & 0 \end{bmatrix}$, where $e \in \R^2$ and $I \in \cS^2$ denotes the identity matrix. $F$ has four vertices given by 
\begin{equation} \label{box_vertices}
v^1 = \begin{bmatrix} 0 \\ 0 \end{bmatrix}, \quad v^2 = \begin{bmatrix} 0 \\ 1 \end{bmatrix}, \quad v^3 = \begin{bmatrix} 1 \\ 0 \end{bmatrix}, \quad v^4 = \begin{bmatrix} 1 \\ 1 \end{bmatrix}.
\end{equation}
The feasible region of the RLT relaxation is given by
\[
\cF = \left\{(x,X) \in \R^2 \times \cS^2: \begin{array}{rcl} x & \leq & e \\ x & \geq & 0 \\ \begin{bmatrix} X - xe^T - ex^T  + ee^T & e x^T - X  \\ x e^T - X & X \end{bmatrix}
& \geq & 0 \end{array} \right\}. 
\]
By Proposition~\ref{vertex1}, there are four vertices of $\cF$ in the form of $(v^j,v^j (v^j)^T),~j = 1,\ldots,4$. Similarly, by Proposition~\ref{midpoint-vertex}, $\cF$ has another set of six vertices in the form of
\[
\left(\textstyle\frac{1}{2}(v^i + v^j), \textstyle\frac{1}{2} (v^i (v^j)^T + v^j (v^i)^T)\right),~1 \leq i < j \leq 4.
\]
We now claim that $\cF$ has at least one other vertex that does not belong to these two sets. Consider $(\hat x, \hat X) \in \R^2 \times \cS^2$ given by
\begin{equation} \label{different_vertex}
\hat x = \begin{bmatrix} \textstyle\frac{1}{2} \\ \textstyle\frac{1}{2} \end{bmatrix}, \quad \hat X = \begin{bmatrix} \textstyle\frac{1}{2} & 0 \\ 0 & 0 \end{bmatrix}.
\end{equation}
It is easy to verify that $(\hat x, \hat X) \in \cF$ and that $(\hat x, \hat X)$ does not belong to either of the two sets of vertices identified by Proposition~\ref{vertex1} and Proposition~\ref{midpoint-vertex}. It is an easy exercise to show that $(\hat x, \hat X) \pm (\hat d, \hat D) \in \cF$ if and only if $(\hat d, \hat D) = (0,0)$. Therefore, we conclude that $(\hat x, \hat X)$ is a vertex of $\cF$ by Lemma~\ref{vertex_chars}~(ii). 
\end{example}

By Example~\ref{Example1}, the two sets of vertices identified in Propositions~\ref{vertex1} and \ref{midpoint-vertex} do not necessarily encompass all vertices of $\cF$ in general even if $F$ is a polytope. In the next section, we identify a subclass of quadratic programs for which all vertices of $\cF$ are completely characterized by Propositions~\ref{vertex1} and \ref{midpoint-vertex}. We then discuss the implications of this observation on general quadratic programs.

\subsubsection{A Specific Class of Quadratic Programs} \label{specific_QP_class}

In this section, we present a specific class of quadratic programs with the property that all vertices of the feasible region $\cF$ of the RLT relaxation are precisely given by the union of the two sets identified in Propositions~\ref{vertex1} and \ref{midpoint-vertex}. 

Consider the class of instances of (QP), where $Q \in \cS^n$, $c \in \R^n$, and 
\begin{equation} \label{specific_QP}
H = a \in \R^n_+ \backslash \{0\}, \quad h = 1, \quad G = - I \in \cS^n, \quad g = 0 \in \R^n. 
\end{equation}
Therefore, the feasible region is given by
\begin{equation} \label{specific_F}
F = \left\{x \in \R^n: a^T x = 1, \quad x \geq 0\right\}.
\end{equation}

Let us define the following index sets:
\begin{eqnarray}
 \mathbf{P} & = & \left\{j \in \{1,\ldots,n\}: a_j > 0\right\}, \label{def_ind_P}\\
 \mathbf{Z} & = & \left\{j \in \{1,\ldots,n\}: a_j = 0\right\}. \label{def_ind_Z}
\end{eqnarray}

It is straightforward to verify that the set of vertices of $F$ is
\begin{equation} \label{def_specific_V}
V = \left\{\left(\textstyle\frac{1}{a_j}\right) e^j: j \in  \mathbf{P}\right\}.    
\end{equation}

The feasible region $\cF$ of the corresponding RLT relaxation is given by 
\begin{equation} \label{def_specific_cF}
\cF = \left\{(x,X) \in \R^n \times \cS^n: X a = x, \quad a^T x = 1, \quad x \geq 0, \quad X \geq 0 \right\}.  
\end{equation}

We next present our main result in this section.

\begin{proposition} \label{complete_vertices}
Suppose that $\cF$ is given by \eqref{def_specific_cF}, where $a \in \R^n_+ \backslash \{0\}$. Then, $(\hat x, \hat X)$ is a vertex of $\cF$ if and only if $(\hat x, \hat X) = (v,vv^T)$ for some $v \in V$, where $V$ is given by \eqref{def_specific_V}, or $(\hat x, \hat X) = \left( \textstyle\frac{1}{2}(v^1 + v^2), \textstyle\frac{1}{2}(v^1 (v^2)^T + v^2 (v^1)^T)\right)$ for some $v^1 \in V$, $v^2 \in V$, and $v^1 \neq v^2$.
\end{proposition}
\begin{proof}
By Propositions~\ref{vertex1} and \ref{midpoint-vertex}, it suffices to prove the forward implication. 

Let us first define
\begin{eqnarray} 
w^k & = & \left(\textstyle\frac{1}{a_k}\right) e^k, \quad k \in \mathbf{P}, \label{tempdef1} \\
W^k & = & w^k (w^k)^T = \left(\textstyle\frac{1}{a_k^2}\right) e^k (e^k)^T, \quad k \in \mathbf{P}, \label{tempdef2} \\
z^{ij} & = & \textstyle\frac{1}{2} \left(w^i + w^j \right) =  \left(\textstyle\frac{1}{2 a_i}\right) e^i + \left(\textstyle\frac{1}{2 a_j}\right) e^j, \quad i \in \mathbf{P}, j \in \mathbf{P}, i \neq j, \label{tempdef3} \\ 
Z^{ij} & = & \textstyle\frac{1}{2} \left(w^i (w^j)^T + w^j (w^i)^T\right) = \left(\textstyle\frac{1}{2 a_i a_j}\right) \left(e^i (e^j)^T + e^j (e^i)^T\right). \quad i \in \mathbf{P}, j \in \mathbf{P}, i \neq j, \label{tempdef4} 
\end{eqnarray}
where $\mathbf{P}$ is given by \eqref{def_ind_P}.
By \eqref{def_specific_V} and Propositions~\ref{vertex1} and \ref{midpoint-vertex}, it follows that each of $(w^k,W^k),~k \in \mathbf{P}$, and $(z^{ij},Z^{ij}),~i \in \mathbf{P}, ~j \in \mathbf{P}, ~i \neq j$, is a vertex of $\cF$. 

Let $(\hat x, \hat X)$ be a vertex of $\cF$. By \eqref{def_specific_cF} and \eqref{def_ind_P}, we obtain
\begin{equation} \label{feas_imps}
\hat X_{\mathbf{P}\mathbf{P}} \, a_{\mathbf{P}} = \hat x_{\mathbf{P}}, \quad a_{\mathbf{P}}^T \hat x_{\mathbf{P}} = 1, \quad \hat x \geq 0, \quad \hat X \geq 0.
\end{equation}
First, we claim that $\hat x_{\mathbf{Z}} = 0$, $\hat X_{\mathbf{P}\mathbf{Z}} = 0$, $\hat X_{\mathbf{Z}\mathbf{P}} = 0$, and $\hat X_{\mathbf{Z}\mathbf{Z}} = 0$. Indeed, by \eqref{feas_imps}, if any of these conditions is not satisfied, it is easy to construct a nonzero $(\hat d, \hat D) \in \R^n \times \cS^n$ such that $(\hat x, \hat X) \pm (\hat d, \hat D) \in \cF$, which would contradict that $(\hat x, \hat X)$ is a vertex of $\cF$ by Lemma~\ref{vertex_chars}~(ii). Therefore, 
\begin{eqnarray*}
\hat X & = & \sum\limits_{k \in \mathbf{P}} \hat X_{kk} e^k (e^k)^T + \textstyle\frac{1}{2} \sum\limits_{i \in \mathbf{P}} \sum\limits_{j \in \mathbf{P}: j \neq i} \hat X_{ij} \left(e^i (e^j)^T + e^j (e^i)^T\right) \\
 & = & \sum\limits_{k \in \mathbf{P}} \mu_k W^k +  \sum\limits_{i \in \mathbf{P}} \sum\limits_{j \in \mathbf{P}: j \neq i} \lambda_{ij} Z^{ij},
\end{eqnarray*}
where $\mu_k = \hat X_{kk} \, a_k^2  \geq 0,~k \in \mathbf{P}$; $\lambda_{ij} =  \hat X_{ij} \, a_i a_j \geq 0,~i \in \mathbf{P},~j \in \mathbf{P},~i \neq j$; $W^k$ and $Z^{ij}$ are defined as in \eqref{tempdef2} and \eqref{tempdef4}, respectively. By using $W^k a = w^k,~k \in \mathbf{P}$; $Z^{ij} a = z^{ij},~i \in \mathbf{P},~j \in \mathbf{P},~i \neq j$; and \eqref{def_specific_cF}, the previous equality implies that 
\[
\hat x = \hat X a = \sum\limits_{k \in \mathbf{P}} \mu_k w^k +  \sum\limits_{i \in \mathbf{P}} \sum\limits_{j \in \mathbf{P}: j \neq i} \lambda_{ij} z^{ij}. 
\]

By \eqref{feas_imps}, we obtain
\[
a_{\mathbf{P}}^T \hat X_{\mathbf{P}\mathbf{P}} \, a_{\mathbf{P}} = \sum\limits_{k \in \mathbf{P}} \hat X_{kk} \, a_k^2 +  \sum\limits_{i \in \mathbf{P}} \sum\limits_{j \in \mathbf{P}: j \neq i} \hat X_{ij} \, a_i a_j = \sum\limits_{k \in \mathbf{P}} \mu_k +  \sum\limits_{i \in \mathbf{P}} \sum\limits_{j \in \mathbf{P}: j \neq i} \lambda_{ij} = 1.
\]

Therefore, $(\hat x, \hat X)$ is given by a convex combination of $(w^k,W^k),~k \in \mathbf{P}$, and $(z^{ij},Z^{ij}),~i \in \mathbf{P}, ~j \in \mathbf{P}, ~i \neq j$. By Propositions~\ref{vertex1} and \ref{midpoint-vertex}, we conclude that either $(\hat x, \hat X) = (w^k,W^k)$ for some $k \in \mathbf{P}$, or $(\hat x, \hat X) = (z^{ij},Z^{ij})$ for some $i \in \mathbf{P}, ~j \in \mathbf{P}, ~i \neq j$. This completes the proof.
\end{proof}

Our next result gives a closed-form expression of the lower bound $\ell_R$ for an instance of (QP) in this specific class.

\begin{corollary} \label{RLT_specific_closed_form}
Consider an instance of (QP), where $F$ is given by \eqref{specific_F} and $a \in \R^n_+ \backslash \{0\}$. If $\ell^*_R$ is finite, then
\begin{equation} \label{upper_bound_lR}
\ell^*_R = \min\left\{\min\limits_{v \in V}\textstyle\left\{\frac{1}{2} v^T Q v + c^T v\right\}, \min\limits_{v^1 \in V,~v^2 \in V,~v^1 \neq v^2}
\frac{1}{2}\left((v^1)^T Q v^2 + c^T (v^1 + v^2)\right)\right\}, 
\end{equation}
where $V$ and $\mathbf{P}$ are given by \eqref{def_specific_V} and \eqref{def_ind_P}, respectively. 
\end{corollary}
\begin{proof}
If $\ell_R$ is finite, the relation \eqref{upper_bound_lR} follows from Proposition~\ref{complete_vertices} since (RLT) is a linear programming problem and the optimal value is attained at a vertex.
\end{proof}

We close this section with a discussion of a well-studied class of quadratic programs that belong to the specific class of quadratic programs identified in this section. An instance of (QP) is referred to as a \emph{standard quadratic program} (see, e.g.,~\cite{Bomze98}) if 
\begin{equation} \label{StQP}
H = e \in \R^n, \quad h = 1, \quad G = - I \in \cS^n, \quad g = 0 \in \R^n. 
\end{equation}
Therefore, the feasible region of a standard quadratic program is the unit simplex given by 
\begin{equation} \label{StQP_F}
F = \left\{x \in \R^n: e^T x = 1, \quad x \geq 0\right\}.
\end{equation}
Similarly, the feasible region of the RLT relaxation of a standard quadratic program is given by 
\begin{equation} \label{StQP_cF}
\cF = \left\{(x,X) \in \R^n \times \cS^n: X e = x, \quad e^T x = 1, \quad x \geq 0, \quad X \geq 0 \right\}.  
\end{equation}

Proposition~\ref{complete_vertices} gives rise to the following result on standard quadratic programs.

\begin{corollary} \label{StQP_vertices}
Consider an instance of a standard quadratic program and let 
$\cF$ denote the feasible region of the RLT relaxation given by \eqref{StQP_cF}. Then, $(\hat x, \hat X)$ is a vertex of $\cF$ if and only if $(\hat x, \hat X) = (e^j,e^j (e^j)^T)$ for some $j = 1,\ldots,n$, or $(\hat x, \hat X) = \left( \textstyle\frac{1}{2}(e^i + e^j), \textstyle\frac{1}{2}\left(e^i (e^j)^T + e^j (e^i)^T\right)\right)$ for some $1 \leq i < j \leq n$. Furthermore, 
\[
\ell^*_R = \min\left\{\min\limits_{k = 1,\ldots,n}\left\{\textstyle\frac{1}{2} Q_{kk} + c_k \right\}, \min\limits_{i=1,\ldots,n;~j = 1,\ldots,n;~i \neq j}
\textstyle\frac{1}{2}\left(Q_{ij} + c_i + c_j \right)\right\}.
\]
\end{corollary}
\begin{proof}
The first assertion follows from Proposition~\ref{complete_vertices} and \eqref{def_specific_V} by using $a = e \in \R^n_+ \backslash \{0\}$, and the second one from Lemma~\ref{bounded_imp} and Corollary~\ref{RLT_specific_closed_form} since $F$ is bounded.
\end{proof}

In~\cite{SagolY15}, using an alternative copositive formulation of standard quadratic programs in~\cite{BomzeDKRQT00}, 
a hierarchy of linear programming relaxations arising from the sequence of polyhedral approximations of the copositive cone proposed by~\cite{KlerkP02} was considered and the same lower bound given by Corollary~\ref{StQP_vertices} was established for the first level of the hierarchy. Therefore, it is worth noting that the relaxation arising from the copositive formulation turns out to be equivalent to the RLT relaxation arising from the usual formulation of standard quadratic programs as an instance of (QP). 

\subsubsection{Implications on General Quadratic Programs}

In Section~\ref{specific_QP_class}, we identified a specific class of quadratic programs with the property that Propositions~\ref{vertex1} and \ref{midpoint-vertex} completely characterize the set of all vertices of the feasible region of the corresponding RLT relaxation. In this section, we first observe that every quadratic program can be equivalently formulated as instance of (QP) in this class. We then discuss the implications of this observation on RLT relaxations of general quadratic programs.

Consider a general quadratic program, where $F$ given by \eqref{def_F} is nonempty. By Lemma~\ref{decomp_result} and \eqref{def_Dinf_alt}, 
\begin{equation} \label{decomp_polyhedra_2}
F = \textrm{conv}\left(\left\{v^1,\ldots,v^s\right\}\right) + \textrm{cone}\left(\left\{d^1,\ldots,d^t\right\}\right),    
\end{equation}
where $v^i \in F_i,~i = 1,\ldots,s$, and each $F_i \subseteq F,~i = 1,\ldots,s$, denotes a minimal face of $F$, and $d^1,\ldots,d^t$ are the generators of $F_\infty$. Let us define 
\begin{equation} \label{matrix_defs}
M = \begin{bmatrix}  v^1 & \cdots & v^s\end{bmatrix} \in \R^{n \times s}, \quad P = \begin{bmatrix}  d^1 & \cdots & d^{t} \end{bmatrix} \in \R^{n \times t}.
\end{equation}
By \eqref{decomp_polyhedra_2} and \eqref{matrix_defs}, $\hat x \in F$ if and only if there exists $y \in \R^s_+$ and $z \in \R^{t}_+$ such that $e^T y = 1$ and $\hat x = My + Pz$. Therefore, (QP) admits the following alternative formulation:
\[
\textrm{(QPA)} \quad \min\limits_{y \in \R^s, z \in \R^{t}} \left\{ \textstyle\frac{1}{2}\left(\left(My + Pz\right)^T Q \left(My + Pz\right)\right) + c^T (My + Pz): e^T y = 1, \quad y \geq 0, \quad z \geq 0\right\}.
\]
We conclude that every quadratic program admits an equivalent reformulation as an instance in the specific class identified in Section~\ref{specific_QP_class}. However, we remark that this equivalence is mainly of theoretical interest since such a reformulation requires the enumeration of all minimal faces of $F$ and all generators of $F_\infty$, each of which may have an exponential size. 

Nevertheless, in this section, we will discuss the relations between the RLT relaxation of (QP) and that of the alternative formulation (QPA) and draw some conclusions. 

Let us introduce the following notations:
\begin{eqnarray}
    n_A & = & s + t \label{tildedef0} \\
    Q_A & = & \begin{bmatrix} M^T Q M & M^T Q P \\ P^T Q M & P^T Q P \end{bmatrix} \in \cS^{n_A} \label{tildedef1} \\
    c_A & = & \begin{bmatrix} M^T c \\ P^T c \end{bmatrix} \label{tildedef2} \in \R^{n_A} \\
    a_A & = & \begin{bmatrix} e \\ 0 \end{bmatrix} \in \R^{n_A}\label{tildedef3} \\
    x_A & = & \begin{bmatrix} y \\ z \end{bmatrix} \in \R^{n_A}\label{tildedef6} 
\end{eqnarray}

Therefore, (QPA) can be expressed by
\[
\textrm{(QPA)} \quad \min\limits_{x_A \in \R^{{n_A}}} \left\{ \textstyle\frac{1}{2} (x_A)^T Q_A x_A + (c_A)^T x_A: x_A \in F_A\right\},
\]
where 
\begin{equation} \label{def_FA}
F_A = \left\{x_A \in \R^{n_A}: (a_A)^T x_A = 1, \quad x_A \geq 0 \right\}.
\end{equation}

Similarly, the RLT relaxation of (QPA) is given by
\[
\textrm{(RLTA)} \quad \ell^*_{RA} = \min\limits_{x_A \in \R^{n_A}, X_A \in \cS^{n_A}} \left\{\displaystyle\frac{1}{2}\langle Q_A, X_A \rangle + (c_A)^T x_A : (x_A,X_A) \in \cF_A\right\},
\]
where 
\begin{equation} \label{def_cFA}
\cF_A = \left\{(x_A,X_A) \in \R^{n_A} \times \cS^{n_A}: X_A \, a_A = x_A, \quad (a_A)^T x_A = 1, \quad x_A \geq 0, \quad X_A \geq 0 \right\}.
\end{equation}

We now present the first relation between the RLT relaxations of (QP) and (QPA) given by (RLT) and (RLTA), respectively.

\begin{proposition} \label{RLT_comp1}
Consider a general quadratic program, where $F$ given by \eqref{def_F} is nonempty. Then, $\ell^*_R \leq \ell^*_{RA} \leq \ell^*$.    
\end{proposition}
\begin{proof}
Let $(\hat x_A, \hat X_A) \in \cF_A$ be an arbitrary feasible solution of (RLTA). 
 We will construct a corresponding feasible solution $(\hat x, \hat X) \in \cF$ of (RLT) with the same objective function value. Let 
 \begin{equation} \label{corr_feas_soln}
 \hat x = \begin{bmatrix}
     M & P 
 \end{bmatrix} \hat x_A \in \R^n, \quad \hat X = \begin{bmatrix}
     M & P 
 \end{bmatrix} \hat X_A \begin{bmatrix}
     M & P 
 \end{bmatrix}^T \in \cS^n,
 \end{equation}
where $M$ and $P$ are defined as in \eqref{matrix_defs}. By \eqref{decomp_polyhedra_2} and \eqref{tildedef3}, we conclude that $\hat x \in F$, i.e., $G^T \hat x \leq g$ and $H^T \hat x = h$. 
 
 Since $G^T v^i \leq g$ for each $i = 1,\ldots,s$, and $G^T d^j \leq 0$ for each $j = 1,\ldots,t$, we obtain 
\[
G^T \begin{bmatrix}
     M & P 
 \end{bmatrix} - g (a_A)^T \leq 0, 
 \]
where we used \eqref{matrix_defs} and \eqref{tildedef3}. Since $\hat X_A \geq 0$, we have
\begin{eqnarray*}
0 & \leq & \left(G^T \begin{bmatrix}
     M & P 
 \end{bmatrix} - g (a_A)^T\right) \hat X_A \left(G^T \begin{bmatrix}
     M & P 
 \end{bmatrix} - g (a_A)^T\right)^T \\
 & = & G^T \hat X G - G^T \hat x g^T - g \hat x^T G + g g^T,
\end{eqnarray*}
where we used \eqref{def_cFA} and \eqref{corr_feas_soln} in the second line. 

 Since $H^T v^i = h$ for each $i = 1,\ldots,s$, and $H^T d^j = 0$ for each $j = 1,\ldots,t$, we obtain 
\[
H^T \begin{bmatrix}
     M & P 
 \end{bmatrix} = h (a_A)^T, 
 \]
where we used \eqref{matrix_defs} and \eqref{tildedef3}.
Therefore,
\begin{eqnarray*}
H^T \hat X & = & H^T \begin{bmatrix}
     M & P 
 \end{bmatrix} \hat X_A \begin{bmatrix}
     M & P 
 \end{bmatrix}^T  \\ 
 & = & h (a_A)^T \hat X_A \begin{bmatrix}
     M & P 
 \end{bmatrix}^T \\
 & = & h (\hat x_A)^T \begin{bmatrix}
     M & P 
 \end{bmatrix}^T \\
 & = & h \hat x^T,
\end{eqnarray*}
where we used \eqref{def_cFA} and \eqref{corr_feas_soln} in the third line. Therefore, $(\hat x , \hat X) \in \mathcal{F}$. Furthermore, 
\begin{eqnarray*}
\textstyle\frac{1}{2}\langle Q, \hat X \rangle + c^T \hat x & = &  \textstyle\frac{1}{2}\left\langle Q, \begin{bmatrix}
     M & P 
 \end{bmatrix} \hat X_A \begin{bmatrix}
     M & P 
 \end{bmatrix}^T \right\rangle + c^T \begin{bmatrix}
     M & P 
 \end{bmatrix} \hat x_A\\
 & = & \textstyle\frac{1}{2}\left\langle Q_A, \hat X_A  \right\rangle + (c_A)^T \hat x_A,
\end{eqnarray*}
where we used \eqref{corr_feas_soln} in the first line, and \eqref{tempdef2} and \eqref{tempdef3} in the second line. Therefore, for each $(\hat x_A, \hat X_A) \in \cF_A$, there exists a corresponding solution $(\hat x, \hat X) \in \cF$ with the same objective function value. We conclude that $\ell^*_R \leq \ell^*_{RA} \leq \ell^*$. 
\end{proof}

By Proposition~\ref{RLT_comp1}, the RLT relaxation (RLTA) of the alternative formulation (QPA) is at least as tight as the RLT relaxation (RLT) of the original formulation (QP). An interesting question is whether the lower bounds arising from the two relaxations are in fact equal (i.e., $\ell^*_R = \ell^*_{RA}$). Our next example illustrates that this is, in general, not true. 

\begin{example} \label{Example2}
Consider the following instance of (QP) for $n = 2$, where
\[
Q = \begin{bmatrix} -2 & 2 \\ 2 & 2 \end{bmatrix}, \quad c = \begin{bmatrix} 0 \\ -2 \end{bmatrix}, 
\]
and 
\[
F = \left\{x \in \R^2: 0 \leq x_j \leq 1, \quad j = 1,2\right\}, 
\]
i.e., we have $p = 0$, $m = 4$, $G = \begin{bmatrix} I & - I \end{bmatrix}$, and $g^T = \begin{bmatrix} e^T & 0 \end{bmatrix}$, where $e \in \R^2$ and $I \in \cS^2$ denotes the identity matrix. For the RLT relaxation of the original formulation, we obtain $\ell^*_R = -\textstyle\frac{3}{2}$, and an optimal solution is given by \eqref{different_vertex}. By \eqref{box_vertices}, we have $F = \textrm{conv}\left(\{v^1,\ldots,v^4\}\right)$. Defining $M = \begin{bmatrix} v^1 \cdots v^4 \end{bmatrix} \in \R^{2 \times 4}$ and using \eqref{tempdef2}, \eqref{tempdef3}, and \eqref{tempdef4}, we obtain
\[
Q_A = M^T Q M = \begin{bmatrix} 0 & 0 & 0 & 0 \\ 0 & 2 & 2 & 4 \\ 0 & -2 & 2 & 0 \\ 0 & 4 & 0 & 4\end{bmatrix}, \quad c_A = M^T c = \begin{bmatrix} 0 \\ -2 \\ 0 \\ -2 \end{bmatrix}, \quad a_A = \begin{bmatrix} 1 \\ 1 \\ 1 \\ 1 \end{bmatrix}. 
\]
Therefore, (QP) can be equivalently formulated as (QPA), which is an instance of a standard quadratic program. By Corollary~\ref{StQP_vertices}, we have $\ell^*_{RA} = -1$ and an optimal solution is 
\[
\hat x_A = \begin{bmatrix} 0 \\ 0 \\ \textstyle\frac{1}{2} \\ \textstyle\frac{1}{2} \end{bmatrix}, \quad \hat X_A = \begin{bmatrix} 0 & 0 & 0 & 0 \\ 0 & 0 & 0 & 0 \\ 0 & 0 & 0 & \textstyle\frac{1}{2} \\ 0 & 0 & \textstyle\frac{1}{2} & 0 \end{bmatrix}.
\]
Therefore, we obtain $\ell^*_R = -\textstyle\frac{3}{2} < -1 = \ell^*_{RA}$. In fact, for this instance of (QP), we have $\ell^* = -1$, which is attained at $x^* = \begin{bmatrix} 1 & 0 \end{bmatrix}^T$. Therefore, the RLT relaxation (RLTA) of the alternative formulation (QPA) is not only tighter than that of the original formulation but is, in fact, an exact relaxation.
\end{example}

As illustrated by Example~\ref{Example2}, despite the fact that (QP) and (QPA) are equivalent formulations, (RLTA) may lead to a strictly tighter relaxation of (QP) than (RLT). Therefore, the quality of the RLT relaxation may depend on the particular formulation. Recall, however, that the alternative formulation (QPA) may, in general, have an exponential size. We close this section with the following result on the RLT relaxation of the original formulation.

\begin{corollary} \label{RLT_upper_bound}
Consider a general quadratic program, where $F$ given by \eqref{def_F} is nonempty. Let $v^i \in F_i,~i = 1,\ldots,s$, where each $F_i \subseteq F,~i = 1,\ldots,s$, denotes a minimal face of $F$. Then, 
\[
\ell^*_R \leq \min\left\{\min\limits_{k = 1,\ldots,s}\textstyle\left\{\frac{1}{2} (v^k)^T Q v^k + c^T v^k\right\}, \min\limits_{1 \leq i < j \leq s}
\frac{1}{2}\left((v^i)^T Q v_j + c^T (v^i + v^j)\right)\right\}.
\]   
\end{corollary}
\begin{proof}
If $F$ contains at least one vertex, then the assertion follows from Propositions~\ref{vertex1} and \ref{midpoint-vertex} since each $v^i,~i = 1,\ldots,s$, is a vertex. Otherwise, by \eqref{def_specific_V}, 
the vertices of $F_A$ given by \eqref{def_FA} are $e^j \in \R^{n_A},~j = 1,\ldots,s$. The assertion follows directly from Proposition~\ref{RLT_comp1} and Corollary~\ref{RLT_specific_closed_form} by observing that $(e^i)^T Q_A e^j = (v^i)^T Q v^j$ for each $i = 1,\ldots,s$ and each $j = 1,\ldots,s$ by \eqref{matrix_defs} and \eqref{tildedef1}. 
\end{proof}

\section{Duality and Optimality Conditions} \label{Sec4}

In this section, we focus on the dual problem of (RLT) and discuss optimality conditions.

By defining the dual variables $(u,w,R,S) \in \R^m \times \R^p \times \R^{p \times n} \times \cS^m$ corresponding to the four sets of constraints in \eqref{def_cF}, respectively, the dual of (RLT) is given by
\[
\begin{array}{llrcl}
\textrm{(RLT-D)} & \max\limits_{u \in \R^m, w \in \R^p, R \in \R^{p \times n}, S \in \cS^m} & -u^T g + w^T h - \displaystyle\frac{1}{2}g^T S g & & \\
 & \textrm{s.t.} & & & \\
 & & -G u + H w - R^T h - G S g & = & c\\
 & & R^T H^T + H R + G S G^T & = & Q \\
 & & S & \geq & 0 \\
 & & u & \geq & 0.
\end{array}
\]

Note that the variable $S \in \cS^m$ is scaled by a factor of $\textstyle\frac{1}{2}$ in (RLT-D). We first review the optimality conditions. 

\begin{lemma} \label{opt_cond}
Suppose that (QP) has a nonempty feasible region. 
Then, $(\hat x, \hat X) \in \cF$ is an optimal solution of (RLT) if and only if there exists 
$(\hat u, \hat w, \hat R, \hat S) \in \R^m_+ \times \R^p \times \R^{p \times n} \times \cN^m$ such that
\begin{eqnarray} \label{opt_form}
c & = & -G \hat u + H \hat w - \hat R^T h - G \hat S g, \label{opt_c} \\
Q & = & \hat R^T H^T + H \hat R + G \hat S G^T, \label{opt_Q} \\
\hat u^T \hat r & = & 0, \label{opt_cs1} \\
\left\langle \hat S, G^T \hat X G + \hat r g^T + g \hat r^T - gg^T \right\rangle & = & 0, \label{opt_cs2}
\end{eqnarray}
where $\hat r = g - G^T \hat x \in \R^m_+$.
\end{lemma}
\begin{proof}
Since $(\hat x, \hat X) \in \cF$, we have
\[
G^T \hat X G - G^T \hat x g^T - g \hat x^T G + g g^T = G^T \hat X G + \hat r g^T + g \hat r^T - g g^T \geq 0,
\]
where we used $\hat r = g - G^T \hat x$. The claim now follows from the optimality conditions for (RLT) and (RLT-D).    
\end{proof}

We remark that Lemma~\ref{opt_cond} gives a recipe for constructing instances of (QP) with a known optimal solution of (RLT) and a finite RLT lower bound on the optimal value. We will discuss this further in Section~\ref{Sec6}.

By Lemma~\ref{bounded_imp}, if $F$ is nonempty and bounded, then $\cF$ is nonempty and bounded, which implies that (RLT) has a finite optimal value. By Lemma~\ref{opt_cond}, we conclude that the (RLT-D) always has a nonempty feasible region under this assumption. 

For the first set of vertices of $\cF$ given by Proposition~\ref{vertex1}, we next establish necessary and sufficient optimality conditions.

\begin{proposition} \label{type1vertices}
Suppose that $v \in F$ is a vertex. Suppose that $G = [G^0 \quad G^1]$ so that $(G^0)^T v = g^0$ and $(G^1)^T v < g^1$, where $G^0 \in \R^{n \times m_0}$, $G^1 \in \R^{n \times m_1}$, $g^0 \in \R^{m_0}$, and $g^1 \in \R^{m_1}$. Then, $(v,v v^T) \in \cF$ is an optimal solution of (RLT) if and only if there exists $(\hat u, \hat w, \hat R, \hat S) \in \R^m \times \R^p \times \R^{p \times n} \times \cS^m$, where $\hat u \in \R^m$ and $\hat S \in \cS^m$ can be accordingly partitioned as
\begin{equation} \label{defuStype1}
\hat u = \begin{bmatrix} \hat u^0 \\ 0 \end{bmatrix} \in \R^m_+, \quad \hat S = \begin{bmatrix}\hat S^{00} & \hat S^{01} \\ (\hat S^{01})^T & 0\end{bmatrix} \in \cN^m,
\end{equation}
where $\hat u^0 \in \R^{m_0}$, $\hat S^{00} \in \cS^{m_0}$, and $\hat S^{01} \in \R^{m_0 \times m_1}$, such that \eqref{opt_c} and \eqref{opt_Q} are satisfied. Furthermore, if $\hat S^{00} \in \cS^{m_0}$ is strictly positive and $\hat u_0 \in \R^{m_0}$ is strictly positive, then $(v,vv^T) \in \cF$ is the unique optimal solution of (RLT).  
\end{proposition} 
\begin{proof}
Suppose that $v \in \R^n$ is a vertex of $F$. By Proposition~\ref{vertex1}, $(v, v v^T)$ is a vertex of $\cF$. Let $r^1 = g^1 - (G^1)^T v > 0$. Then, 
\[
\begin{bmatrix} (G^0)^T \\ (G^1)^T \end{bmatrix} v v^T  \begin{bmatrix} (G^0)^T \\ (G^1)^T \end{bmatrix}^T - \begin{bmatrix} (G^0)^T \\ (G^1)^T \end{bmatrix} v \begin{bmatrix} g^0 \\ g^1 \end{bmatrix}^T - \begin{bmatrix} g^0 \\ g^1 \end{bmatrix} v^T \begin{bmatrix} (G^0)^T \\ (G^1)^T \end{bmatrix}^T + g g^T = \begin{bmatrix} 0 & 0 \\ 0 & r^1 (r^1)^T \end{bmatrix}. 
\]
The first assertion now follows from Lemma~\ref{opt_cond}.  

For the second part, suppose further that $\hat S_{00} \in \cS^{m_0}$ is strictly positive and $\hat u_0 \in \R^{m_0}$ is strictly positive. Let $(\hat x, \hat X) \in \cF$ be an arbitrary feasible solution of (RLT). Then, using the same partitions of $G$ and $g$ as before, we have 
\begin{eqnarray*}
(G^0)^T \hat x & = & g^0 - r^0 \\
(G^1)^T \hat x & = & g^1 - r^1 \\
H^T \hat x & = & h \\
H^T \hat X & = & h \hat x^T\\
(G^i)^T \hat X G^j & \geq & -r^i (g^j)^T - g^i (r^j)^T + g^i (g^j)^T, \quad (i,j) \in \{(0,0),(0,1),(1,1)\}, 
\end{eqnarray*}
where $r^0 \in \R^{m_0}_+$ and $r^1 \in \R^{m_1}_+$. By Lemma~\ref{opt_cond}, $(\hat x, \hat X) \in \cF$ is an optimal solution if and only if $r^0 = 0$ and $(G^0)^T \hat X G^0 - g^0 (g^0)^T = 0$. Note that any $(\hat x, \hat X) \in \cF$ with this property should satisfy the following equation:
\[
\begin{bmatrix} (G^0)^T \\ H^T \end{bmatrix} \hat X \begin{bmatrix} (G^0)^T \\ H^T \end{bmatrix}^T = \begin{bmatrix} g^0 \\ h \end{bmatrix} \begin{bmatrix} g^0 \\ h \end{bmatrix}^T.
\]
By Lemma~\ref{matrix-equation}, it follows that $\hat X = vv^T$ is the only solution to this system since $\begin{bmatrix} G^0 & H \end{bmatrix}$ has full row rank by Lemma~\ref{vertex_chars}~(iii). By Lemma~\ref{opt_cond}, $(v,vv^T) \in \cF$ is the unique optimal solution of (RLT).
\end{proof}

Next, we present necessary and sufficient optimality conditions for the second set of vertices of $\cF$ given by Proposition~\ref{midpoint-vertex}.

\begin{proposition} \label{type2vertices}
Let $v^1 \in F$ and $v^2 \in F$ be two vertices such that $v^1 \neq v^2$. Suppose that $G = [G^0 \quad G^1 \quad G^2 \quad G^3]$ so that $(G^0)^T v^1 = (G^0)^T v^2 = g^0$;  $(G^1)^T v^1 = g^1$ and $(G^1)^T v^2 < g^2$; $(G^2)^T v^1 < g^1$ and $(G^2)^T v^2 = g^2$; $(G^3)^T v^1 < g^3$ and $(G^3)^T v^2 < g^3$, 
where $G^0 \in \R^{n \times m_0}$, $G^1 \in \R^{n \times m_1}$, $G^2 \in \R^{n \times m_2}$, $G^3 \in \R^{n \times m_3}$, $g^0 \in \R^{m_0}$, $g^1 \in \R^{m_1}$, $g^2 \in \R^{m_2}$, and $g^3 \in \R^{m_3}$. Then, $(\frac{1}{2}(v^1 + v^2), \frac{1}{2} (v^1 (v^2)^T + v^2 (v^1)^T)) \in \cF$ is an optimal solution of (RLT) if and only if there exists $(\hat u, \hat w, \hat R, \hat S) \in \R^m \times \R^p \times \R^{p \times n} \times \cS^m$, where $\hat u \in \R^m$ and $\hat S \in \cS^m$ can be accordingly partitioned as
\begin{equation} \label{defuStype2}
\hat u = \begin{bmatrix} \hat u^0 \\ 0 \\ 0 \\ 0 \end{bmatrix} \in \R^m_+, \quad \hat S = \begin{bmatrix}\hat S^{00} & \hat S^{01} & \hat S^{02} & \hat S^{03} \\ (\hat S^{01})^T & \hat S^{11} & 0 & 0 \\ (\hat S^{02})^T & 0 & \hat S^{22} & 0 \\ (\hat S^{03})^T & 0 & 0 & 0 \end{bmatrix} \in \cS^m,
\end{equation}
where $\hat u^0 \in R^{m_0}$, $\hat S^{kk} \in \cS^{m_k},~k = 0, 1, 2$, $\hat S^{0j} \in \R^{m_0 \times m_j},~j = 1, 2, 3$, such that \eqref{opt_c} and \eqref{opt_Q} are satisfied. Furthermore, if each of $\hat S^{00} \in \cS^{m_0}$, $\hat S^{01} \in \R^{m_0 \times m_1}$, $\hat S^{02} \in \R^{m_0 \times m_2}$, $\hat S^{11} \in \cS^{m_1}$, $\hat S^{22} \in \cS^{m_2}$, and $\hat u^0 \in \R^{m_0}$ is strictly positive, then $(\frac{1}{2}(v^1 + v^2), \frac{1}{2} (v^1 (v^2)^T + v^2 (v^1)^T)) \in \cF$ is the unique optimal solution of (RLT).
\end{proposition}
\begin{proof} 
By Proposition~\ref{midpoint-vertex}, $(\frac{1}{2}(v^1 + v^2), \frac{1}{2} (v^1 (v^2)^T + v^2 (v^1)^T))$ is a vertex of $\cF$. The proof is similar to the proof of Proposition~\ref{type1vertices}. By a similar argument as in the proof of Proposition~\ref{midpoint-vertex}, we have 
\[
G^T \hat X G - G^T \hat x g^T - g \hat x^T G + g g^T 
= \begin{bmatrix} 0 & 0 & 0 & 0 \\
0 & 0 & + & + \\ 0 & + & 0 & + \\ 0 & + & + & + \end{bmatrix},
\]
where we assume that $G = \begin{bmatrix} G^0 & G^1 & G^2 & G^3 \end{bmatrix}$ and $g$ is partitioned accordingly, and $+$ denotes a submatrix with strictly positive entries. The first claim follows from Lemma~\ref{opt_cond}.

For the second assertion, suppose further that each of $\hat S^{00} \in \cS^{m_0}$, $\hat S^{01} \in \R^{m_0 \times m_1}$, $\hat S^{02} \in \R^{m_0 \times m_2}$, $\hat S^{11} \in \cS^{m_1}$, $\hat S^{22} \in \cS^{m_2}$, and $\hat u_0 \in \R^{m_0}$ is strictly positive. Let $(\hat x, \hat X) \in \cF$ be an arbitrary solution. Then, using the same partition of $G$ and $g$, we have 
\begin{eqnarray*}
(G^i)^T \hat x & = & g^i - r^i, \quad i = 0,1,2,3\\
H^T \hat x & = & h \\
H^T \hat X & = & h \hat x^T\\
(G^i)^T \hat X G^j & \geq & -r^i (g^j)^T - g^i (r^j)^T + g^i (g^j)^T, \quad 0 \leq i \leq j \leq 3,
\end{eqnarray*}
where $r^i \in \R^{m_i}_+,~i = 0,1,2,3$, and the last set of inequalities is componentwise. By Lemma~\ref{opt_cond}, $(\hat x, \hat X) \in \cF$ is an optimal solution if and only if 
\begin{eqnarray*}
   r^0 & = & 0 \\
   (G^0)^T \hat X G^0 & = & g^0 (g^0)^T\\
   (G^0)^T \hat X G^1 & = & -g^0 (r^1)^T + g^0 (g^1)^T\\
   (G^0)^T \hat X G^2 & = & -g^0 (r^2)^T + g^0 (g^2)^T\\
   (G^1)^T \hat X G^1 & = & -r^1 (g^1)^T - g^1 (r^1)^T + g^1 (g^1)^T\\
   (G^2)^T \hat X G^2 & = & -r^2 (g^2)^T - g^2 (r^2)^T + g^2 (g^2)^T.
\end{eqnarray*}
Note that $(\frac{1}{2}(v^1 + v^2), \frac{1}{2} (v^1 (v^2)^T + v^2 (v^1)^T)) \in \cF$ with $r^1 = g^1 - \frac{1}{2} (G^1)^T (v^1 + v^2)$ and $r^2 = g^2 - \frac{1}{2} (G^1)^T (v^1 + v^2)$ satisfies this system. Using a similar argument as in the proof of Proposition~\ref{midpoint-vertex}, one can show that this solution is unique. By Lemma~\ref{opt_cond}, we conclude that $(\frac{1}{2}(v^1 + v^2), \frac{1}{2} (v^1 (v^2)^T + v^2 (v^1)^T)) \in \cF$ is the unique optimal solution of (RLT).
\end{proof}

\section{Exact RLT Relaxations} \label{Sec5}

In this section, we present necessary and sufficient conditions in order for an instance of (QP) to admit an exact RLT relaxation. 

First, following~\cite{yildirim2020alternative}, we define the convex underestimator arising from RLT relaxations. To that end, let 
\begin{equation} \label{param_feas_region}
\cF(\hat x) = \{(x,X) \in \cF_{R}: x = \hat x\}, \quad \hat x \in F.
\end{equation}

We next define the following function:
\begin{equation} \label{conv_undest_rlt}
    \ell_R(\hat x) = \min_{x \in \R^n,X \in \cS^n} \left\{ \textstyle\frac{1}{2}\langle Q, X \rangle + c^T x: (x,X) \in \cF(\hat x)\right\}, \quad \hat x \in F.
\end{equation}

By~\cite{yildirim2020alternative}, $\ell_R(\cdot)$ is a convex underestimator of $q(\cdot)$ over $F$, i.e., $\ell_R(\hat x) \leq q(\hat x)$ for each $\hat x \in F$, and 
\begin{equation} \label{conv_underest_2}
\ell^*_R = \min\limits_{x \in F} \ell_R(x).
\end{equation}

By \eqref{conv_undest_rlt}, 
\begin{equation} \label{conv_underest_3}
\ell_R(\hat x) = c^T \hat x + \ell_R^0(\hat x),   
\end{equation}
where
\[
\begin{array}{rrcl}
\textrm{(RLT)}(\hat x) \quad \ell_R^0(\hat x) = \min\limits_{X \in \cS^n} & \frac{1}{2} \langle Q, X \rangle & & \\
 \textrm{s.t.} & & & \\
 & H^T X & = & h \hat x^T \\
 & G^T X G - G^T \hat x g^T - g \hat x^T G + g g^T & \geq & 0.
 \end{array}
 \]
Note that (RLT)$(\hat x)$ has a nonempty feasible region for each $\hat x \in F$ since $\hat X = \hat x \hat x^T$ is a feasible solution. By defining the dual variables $(R,S) \in \R^{p \times n} \times \cS^m$ corresponding to the first and second sets of constraints in (RLT)$(\hat x)$, respectively, the dual of (RLT)$(\hat x)$ is given by
\[
\begin{array}{llrcl}
\textrm{(RLT-D)}(\hat x) & \max\limits_{S \in \cS^m, R \in \R^{p \times n}} & h^T R \hat x + g^T S G \hat x - \textstyle\frac{1}{2}g^T S g & & \\
 & \textrm{s.t.} & & & \\
 & & R^T H^T + H R + G S G^T & = & Q \\
 & & S & \geq & 0. 
\end{array}
\]

We start with a useful result on the convex underestimator $\ell_R(\cdot)$.

\begin{lemma} \label{unb_parametric}
Suppose that $F$ is nonempty. If there exists $\hat x \in F$ such that $\ell_R(\hat x) = -\infty$, then $\ell_R(\tilde x) = -\infty$ for each $\tilde x \in F$. Therefore, $\ell^*_R = -\infty$.
\end{lemma}
\begin{proof}
Suppose that there exists $\hat x \in F$ such that $\ell_R(\hat x) = -\infty$. By linear programming duality, (RLT-D)$(\hat x)$ is infeasible. Therefore, (RLT-D)$(\tilde x)$ is infeasible for each $\tilde x \in F$ since the feasible region of (RLT-D)$(\hat x)$ does not depend on $\hat x \in F$. Since (RLT)$(\tilde x)$ has a nonempty feasible region for each $\tilde x \in F$, (RLT)$(\tilde x)$ is unbounded below. By \eqref{conv_underest_3}, $\ell_R(\tilde x) = -\infty$ for each $\tilde x \in F$. The last assertion simply follows from \eqref{conv_underest_2}.
\end{proof}

We next establish another property of $\ell_R(\cdot)$.
\begin{lemma} \label{piecewise_lin}
Suppose that $F$ is nonempty and there exists $\hat x \in F$ such that $\ell_R(\hat x) > -\infty$. Then, $\ell_R(\cdot)$ is a piecewise linear convex function.
\end{lemma}
\begin{proof}
 Suppose that $F$ is nonempty and there exists $\hat x \in F$ such that $\ell_R(\hat x) > -\infty$. By using a similar argument as in the proof of Lemma~\ref{unb_parametric}, we conclude that $\ell_R(\tilde x) > -\infty$ for each $\tilde x \in F$. By linear programming duality, for each $\hat x \in F$ the optimal value of (RLT-D)$(\hat x)$ equals $\ell^0_R(\hat x)$. By Lemma~\ref{decomp_result}, there exist feasible solutions $(R^i,S^i) \in \R^{p \times n} \times \cS^m,~i = 1,\ldots,s$, of (RLT-D)$(\hat x)$ and a polyhedral cone $\cC \subseteq \cS^m \times \R^{p \times n}$ such that the feasible region of (RLT-D)$(\hat x)$ is given by $\textrm{conv}\{(R^i,S^i):i = 1,\ldots,s\} + \cC$. Since the optimal value of (RLT-D)$(\hat x)$ is finite, it follows that 
 \[
 \ell^0_R(\hat x) = \max\limits_{i = 1,\ldots,s} \left\{h^T R^i \hat x + g^T S^i G \hat x - \textstyle\frac{1}{2}g^T S^i g\right\}.
 \]
 The assertion follows from \eqref{conv_underest_3}. 
\end{proof}

We remark that a similar result was established in~\cite{QY23} for the convex underestimator $\ell_R(\cdot)$ arising from the RLT relaxation of quadratic programs with box constraints. For this class of problems, the hypothesis of 
Lemma~\ref{piecewise_lin} is vacuous. Therefore, Lemma~\ref{piecewise_lin} extends this result to RLT relaxations of all quadratic programs under a mild assumption. 

We next present a necessary and sufficient condition for instances of (QP) that admit an exact RLT relaxation. 

\begin{proposition} \label{exactness_cond}
Suppose that $F$ is nonempty and $\ell^*$ is finite. Then, 
the RLT relaxation given by (RLT) is exact, i.e., $\ell^*_R = \ell^*$, if and only if there exists $v \in F$ that lies on a minimal face of $F$ such that $(v,vv^T) \in \cF$ is an optimal solution of (RLT). Furthermore, in this case, any $\hat x \in F$ that lies on the same minimal face of $F$ is an optimal solution of (QP).    
\end{proposition}
\begin{proof}
Suppose that $F$ is nonempty, $\ell^*$ is finite, and the RLT relaxation given by (RLT) is exact. Since $\ell^*_R = \ell^*$ is finite, the set of optimal solutions of (QP) is nonempty by the Frank-Wolfe theorem~\cite{FW1956}. Therefore,  for any optimal solution $\hat x \in F$ of (QP), $(\hat x, \hat x \hat x^T)$ is an optimal solution of (RLT) since $\ell^* = q(x^*) = \textstyle\frac{1}{2}\langle Q, \hat x \hat x^T \rangle + c^T \hat x = \ell^*_R$. If $\hat x \in F$ already lies on a minimal face of $F$, then we are done. Otherwise, let us define the submatrices $G^0$, $G^1$ and the subvectors $g^0$, $g^1$ such that \eqref{def_G0_G1} holds. Then, there exists $v \in F$ in a minimal face of $F$ such that $(G^0)^T v = g^0$. Let $\hat d = v - \hat x \in \R^n \backslash \{0\}$. Clearly, $(G^0)^T \hat d = 0$ and $H^T \hat d = 0$. Let us define $\hat D = \hat d \hat x^T + \hat x \hat d^T + \hat d \hat d^T \in \cS^n$. By using a similar argument as in the proof of Proposition~\ref{vertex1}, it is easy to verify that
\begin{eqnarray*}
H^T (\hat x \hat x^T + \hat D) & = & h \left(\hat x + \hat d\right)^T\\
(G^0)^T (\hat x \hat x^T + \hat D) G^0 - (G^0)^T \left(\hat x + \hat d\right) (g^0)^T - g^0 \left(\hat x + \hat d\right)^T G^0 + g^0 (g^0)^T& = & 0\\
(G^0)^T (\hat x \hat x^T + \hat D) G^1 - (G^0)^T \left(\hat x + \hat d\right) (g^1)^T - g^0 \left(\hat x + \hat d\right)^T G^1 + g^0 (g^1)^T& = & 0\\
(G^0)^T (\hat d \hat x^T + \hat x \hat d^T) G^1 - (G^0)^T \hat d (g^1)^T - g^0 \hat d^T G^1 & = & 0   
\end{eqnarray*}
Furthermore, since $(G^1)^T \hat x < g^1$ and $(G^1)^T \hat x \hat x^T G^1 - (G^1)^T \hat x (g^1)^T - g^1 \hat x^T G^1 + g^1 (g^1)^T = ((G^1)^T \hat x - g^1) ((G^1)^T \hat x - g^1)^T > 0$, it follows that there exists a real number $\epsilon > 0$ such that $(\hat x, \hat x \hat x^T) + \epsilon (\hat d, \hat D) \in \cF$ and $(\hat x, \hat x \hat x^T) - \epsilon (\hat d, \hat D) \in \cF$. By the optimality of $(\hat x, \hat x \hat x^T)$ for (RLT), we obtain $\textstyle\frac{1}{2}\langle Q, \hat D \rangle + c^T \hat d = 0$. Furthermore, for $\epsilon = 1$, we have 
\[
(\hat x, \hat x \hat x^T) + (\hat d, \hat D) = (v, v v^T) \in \cF,
\]
which implies that $(v, v v^T) \in \cF$ is an optimal solution of (RLT).

Conversely, suppose that there exists $v \in F$ that lies on a minimal face of $F$ such that $(v,vv^T)$ is an optimal solution of (RLT). Then, 
\[
\ell^*_R = \textstyle\frac{1}{2} \langle Q, vv^T \rangle + c^T v = \textstyle\frac{1}{2} v^T Q v + c^T v \geq \ell^*,
\]
where the last inequality follows from $v \in F$. Then, $\ell^*_R = \ell^*$ by \eqref{rlt-lb}. 

For the second assertion, $v \in F$ is clearly an optimal solution of (QP). If $v$ is a vertex, then there is nothing to prove since $v$ is the only point in the minimal face. Otherwise, let  $v^\prime \in F \backslash \{v\}$ lie on the same minimal face of $F$. Then, by defining $\hat d = v^\prime - v \in \R^n \backslash \{0\}$ and $\hat D = \hat d v^T + v \hat d^T + \hat d \hat d^T \in \cS^n$ and using a similar argument as above, we conclude that $\textstyle\frac{1}{2}\langle Q, \hat D \rangle + c^T \hat d = 0$, which implies that $\ell^*_R = \textstyle\frac{1}{2} \langle Q, vv^T \rangle + c^T v = \textstyle\frac{1}{2} \langle Q, v^\prime (v^\prime)^T \rangle + c^T v^\prime$. The second assertion follows.
\end{proof}

Proposition~\ref{exactness_cond} presents a necessary and sufficient condition in order for an instance of (QP) to admit an exact RLT relaxation. This result is a generalization of the corresponding result established for RLT relaxations of quadratic programs with box constraints~\cite{QY23}. In the next section, we discuss the implications of our results on the algorithmic construction of instances of (QP) with exact, inexact, or unbounded RLT relaxations.

\section{Implications on Algorithmic Constructions of Instances} \label{Sec6}

In this section, we discuss how our results can be utilized to design algorithms for constructing an instance of (QP) such that the lower bound from the RLT relaxation and the optimal value of (QP) will have a predetermined relation. In particular, our discussions on instances with exact and inexact RLT relaxations in this section can be viewed as generalizations of the algorithmic constructions discussed in~\cite{QY23} for quadratic programs with box constraints.

To that end, we will assume that the nonempty feasible region $F$ is fixed and given by \eqref{def_F}. We will discuss how to construct an objective function in such a way that the resulting instance of (QP) will have an exact, inexact, or unbounded RLT relaxation.

\subsection{Instances with an Unbounded RLT Relaxation}

By Lemma~\ref{bounded_imp}, if $F$ is nonempty and bounded, then the RLT relaxation cannot be unbounded. Therefore, a necessary condition to have an unbounded RLT relaxation is that $F$ is unbounded. In this case, the recession cone $\cF_\infty$ given by \eqref{def_cDinf} contains a nonzero $(\hat d, \hat D) \in \R^n \times \cS^n$ by Lemma~\ref{rel-rec-cones}. By linear programming duality, the RLT relaxation (RLT) is unbounded if and only if 
\begin{equation} \label{unb_RLT}
\textstyle\frac{1}{2}\langle Q, \hat D \rangle + c^T \hat d < 0, \quad \textrm{for some}~(\hat d, \hat D) \in \R^n \times \cS^n.    
\end{equation}

Let $\hat x \in F$ and $\hat d \in F_\infty$ be arbitrary. By Lemma~\ref{rel-rec-cones}, $(\hat d, \hat D) \in \cF_\infty$, where $\hat D = \hat x \hat d^T + \hat d \hat x^T$. By \eqref{unb_RLT}, it suffices to choose $(Q,c) \in \cS^n \times \R^n$ such that 
\[
\textstyle\frac{1}{2}\langle Q, \hat D \rangle + c^T \hat d = \hat d^T (Q \hat x + c) < 0,
\]
which would ensure that the RLT relaxation is unbounded. 

While this simple procedure can be used to construct an instance of (QP) with an unbounded RLT relaxation, we remark that the resulting instance of (QP) itself may also be unbounded. In particular, if $\hat d^T Q \hat d \leq 0$ in the aforementioned procedure, then (QP) will be unbounded along the direction $\hat x + \lambda \hat d$, where $\lambda \geq 0$. One possible approach to construct an instance of (QP) with a finite optimal value but an unbounded RLT relaxation is to generate $(Q,c) \in \cS^n \times \R^n$ in such a way that \eqref{unb_RLT} holds while a tighter relaxation of (QP) such as the RLT relaxation strengthened by semidefinite constraints has a finite lower bound. This property can be satisfied by ensuring the feasibility of $(Q,c)$ with respect to the dual problem of the tighter relaxation. Such an approach would require the solution of a semidefinite feasibility problem. 

\subsection{Instances with an Exact RLT Relaxation}

By Proposition~\ref{exactness_cond}, the RLT relaxation is exact if and only if there exists $v \in F$ that lies on a minimal face of $F$ such that $(v,vv^T)$ is an optimal solution of (RLT). This result can be used to easily construct an instance of (QP) with an exact RLT relaxation. 

The first step requires the computation of a point $v \in F$ that lies on a minimal face of $F$. Then, defining $(\hat x, \hat X) = (v,vv^T) \in \cF$, it is easy to construct $\hat u \in \R^m_+$ and $\hat S \in \cN^n$ such that \eqref{opt_cs1} and \eqref{opt_cs2} are satisfied. Finally, choosing an arbitrary $(\hat w, \hat R) \in \R^p \times \R^{p \times n}$ and defining $Q$ and $c$ using \eqref{opt_cs1} and \eqref{opt_cs2}, respectively, it follows from Lemma~\ref{opt_cond} that $(v,vv^T) \in \cF$ and $(\hat u, \hat w, \hat R, \hat S) \in \R^m \times \R^p \times \R^{p \times n} \times \cS^m$ are optimal solutions of (RLT) and (RLT-D), respectively. By Proposition~\ref{exactness_cond}, we conclude that the RLT relaxation is exact and that $v \in F$ is an optimal solution of (QP). We remark that this procedure not only ensures an exact RLT relaxation but also yields an instance of (QP) with a predetermined optimal solution $v \in F$.

\subsection{Instances with an Inexact and Finite RLT Relaxation}

First, we assume that $F$ has at least two distinct vertices $v^1 \in F$ and $v^2 \in F$. In this case, one can choose $\hat u \in \R^m_+$ and $\hat S \in \cN^m$ such that the assumptions of the second part of Proposition~\ref{type2vertices} are satisfied. Then, by choosing an arbitrary $(\hat w, \hat R) \in \R^p \times \R^{p \times n}$ and defining $Q$ and $c$ using \eqref{opt_cs1} and \eqref{opt_cs2}, respectively, we obtain that $(\frac{1}{2}(v^1 + v^2), \frac{1}{2} (v^1 (v^2)^T + v^2 (v^1)^T)) \in \cF$ is the unique optimal solution of (RLT). Therefore, the RLT relaxation has a finite lower bound $\ell_R$. By Proposition~\ref{exactness_cond}, we conclude that the RLT relaxation is inexact, i.e., $-\infty < \ell^*_R < \ell^*$.

We next consider the case in which $F$ has no vertices. In this case, $\cF_\infty$ also has no vertices by Lemma~\ref{novertex_imp}. Therefore, the RLT relaxation can never have a unique optimal solution. However, our next result shows that we can extend the procedure above to construct an instance of (QP) with an inexact but finite RLT lower bound under a certain assumption on $F$.

\begin{lemma} \label{inexact_RLT_minimal_faces}
 Suppose that $F$ given by \eqref{def_F} has no vertices but has two distinct minimal faces $F_1 \subseteq F$ and $F_2 \subseteq F$. Let $v^1 \in F_1$ and $v^2 \in F_2$. Suppose that $G = [G^0 \quad G^1 \quad G^2 \quad G^3]$ is defined as in Proposition~\ref{type2vertices}. Let $\hat u \in \R^m_+$ and $\hat S \in \cN^m$ be such that the assumptions of the second part of Proposition~\ref{type2vertices} are satisfied. Let $(\hat w, \hat R) \in \R^p \times \R^{p \times n}$ be arbitrary. If $Q$ and $c$ are defined by \eqref{opt_cs1} and \eqref{opt_cs2}, respectively, then the RLT relaxation of the resulting instance of (QP) is inexact and satisfies $-\infty < \ell^*_R < \ell^*$.
\end{lemma}
\begin{proof}
 Arguing similarly to the first part of the proof of Proposition~\ref{type2vertices}, we conclude that $(\frac{1}{2}(v^1 + v^2), \frac{1}{2} (v^1 (v^2)^T + v^2 (v^1)^T)) \in \cF$ is an optimal solution of the RLT relaxation of the resulting instance of (QP). Therefore, $-\infty < \ell^*_R$.

 Next, we argue that the RLT relaxation is inexact. First, since each of $F_1$ and $F_2$ are minimal faces, they are affine subspaces given by
 \begin{eqnarray*}
 F_1 & = & \{x \in \R^n: (G^0)^T x = g^0, \quad (G^1)^T x = g^1, \quad H^T x = h\},\\
 F_2 & = & \{x \in \R^n: (G^0)^T x = g^0, \quad (G^2)^T x = g^2, \quad H^T x = h\}.
 \end{eqnarray*}
Suppose, for a contradiction, that the RLT relaxation is exact. Then, by Proposition~\ref{exactness_cond}, there exists $v \in F_0$, where $F_0 \subseteq F$ is a minimal face of $F$, such that $(v,vv^T)$ is an optimal solution of (RLT). Let us define $r = g - G^T v \geq 0$ and partition $r$ accordingly as
 \[
 (G^i)^T v = g^i - r^i, \quad i = 0,1,2,3,
 \]
where $r^i \in \R^{m_i}_+,~i = 0,1,2,3$. By Lemma~\ref{opt_cond}, $(v,vv^T)$ and $(\hat u, \hat w, \hat R, \hat S) \in \R^m \times \R^p \times \R^{p \times n} \times \cS^m$ satisfy the optimality conditions \eqref{opt_c}--\eqref{opt_cs2}. By \eqref{opt_cs1} and $\hat u^0 > 0$, we conclude that $r^0 = 0$. On the other hand, 
\[
G^T v v^T G - G^T v g^T - g v^T G + g g^T = G^T v v^T G + r g^T + g \hat r^T - gg^T = r r^T.
\]
By \eqref{opt_cs2}, we obtain $r^1 = 0$ and $r^2 = 0$ since $\hat S^{11} \in \cS^{m_1}$ and $\hat S^{22} \in \cS^{m_2}$ are strictly positive. It follows that $v \in F_0$ satisfies
\[
(G^0)^T v = g^0, \quad (G^1)^T v = g^1, \quad (G^2)^T v = g^2, \quad H^T v = h,
\]
i.e., $v$ lies on a face whose dimension is strictly smaller than each of $F_1$ or $F_2$. This contradicts our assumption that each of $F_1$ and $F_2$ is a minimal face of $F$. We therefore conclude that (RLT) cannot have an optimal solution of the form $(v,vv^T)$, where $v$ lies on a minimal face of $F$. By Proposition~\ref{exactness_cond}, we conclude that the RLT relaxation is inexact, i.e., $-\infty < \ell^*_R < \ell^*$.
\end{proof}

The next example illustrates the algorithmic construction of Lemma~\ref{inexact_RLT_minimal_faces}.

\begin{example} \label{Ex_min_face_RLT}
Suppose that 
\[
F = \{x \in \R^2: x_1 + x_2 \leq 1, \quad x_1 + x_2 \geq -1\},
\]
i.e., $n = 2$, $m = 2$, $p = 0$, and 
\[
G = \begin{bmatrix} 1 & -1 \\ 1 & -1 \end{bmatrix}, \quad g = \begin{bmatrix} 1 \\ 1 \end{bmatrix}.
\]
Note that $F$ has no vertices since it contains the line $\{x \in \R^2: x_1 + x_2 = 0\}$. $F$ has two minimal faces given by
\begin{eqnarray*}
    F_1 & = & \{x \in \R^2: x_1 + x_2 = 1\},\\
    F_2 & = & \{x \in \R^2: x_1 + x_2 = -1\}.
\end{eqnarray*}
Therefore, $F$ satisfies the assumptions of Lemma~\ref{inexact_RLT_minimal_faces}. Let us choose $v^1 = [1 \quad 0]^T \in F_1$ and $v_2 = [0 \quad -1]^T \in F_2$. We obtain
\[
\hat x = \textstyle \frac{1}{2}(v^1 + v^2) = \begin{bmatrix}\textstyle \frac{1}{2} \\ -\textstyle \frac{1}{2} \end{bmatrix}, \quad \hat X = \textstyle \frac{1}{2} (v^1 (v^2)^T + v^2 (v^1)^T) = \begin{bmatrix} 0 & -\textstyle \frac{1}{2}\\ -\textstyle \frac{1}{2} & 0\end{bmatrix}.
\]
Since $g - G^T \hat x > 0$, we choose $\hat u = 0 \in \R^2$. Using the partition given in Proposition~\ref{type2vertices}, we obtain that $G^0$ is an empty matrix, $G^1 = [1 \quad 1]^T$, $G^2 = [-1 \quad -1]^T$, and $G^3$ is an empty matrix. By the second part of Proposition~\ref{type2vertices}, we choose 
\[
\hat S = \begin{bmatrix}
1 & 0 \\ 0 & 2    
\end{bmatrix}.
\]
Therefore, by \eqref{opt_cs1} and \eqref{opt_cs2}, we obtain
\[
Q = \begin{bmatrix}
   3 & 3 \\ 3 & 3 
\end{bmatrix}, \quad c = \begin{bmatrix} 1 \\ 1 \end{bmatrix}.
\]
By Proposition~\ref{type2vertices}, $\ell^*_R = \textstyle\frac{1}{2}\langle Q, \hat X \rangle + c^T \hat x = -\textstyle\frac{3}{2}$. On the other hand, it is easy to verify that $\ell^* = -\textstyle\frac{1}{6}$ and an optimal solution of (QP) is given by $x^* = -\left[\textstyle\frac{1}{6} \quad \textstyle\frac{1}{6}\right]^T$. Therefore, $-\infty < \ell^*_R < \ell^*$. 
\end{example}

\subsection{Implications of One Minimal Face}

We finally consider the case in which $F$ has exactly one minimal face. For any instance of (QP) with this property, we show that the RLT relaxation is either exact or unbounded below.

\begin{lemma} \label{one_min_face}
Consider an instance of (QP), where $F$ given by \eqref{def_F} has exactly one minimal face. Then, the RLT relaxation is either exact or unbounded below.  
\end{lemma}
\begin{proof}
Suppose that $F$ has one minimal face $F_0 \subseteq F$ and let $v \in F_0$. Suppose that $G = [G^0 \quad G^1]$ so that $(G^0)^T v = g^0$ and $(G^1)^T v < g^1$, where $G^0 \in \R^{n \times m_0}$, $G^1 \in \R^{n \times m_1}$, $g^0 \in \R^{m_0}$, and $g^1 \in \R^{m_1}$. First, we claim that the set of inequalities $(G^1)^T x \leq g^1$ is redundant for $F$. Let $\hat x \in F$ be arbitrary. 
By Lemma~\ref{decomp_result}, 
\[
F = \{v\} + F_\infty,
\]
where $F_\infty$ is given by \eqref{def_Dinf}. Therefore, there exists $\hat d \in F_\infty$ such that $\hat x = v + \hat d$. Therefore, $(G^1)^T \hat x = (G^1)^T v + (G^1)^T \hat d < g^1$ since $(G^1)^T v < g^1$ and $(G^1)^T \hat d \leq 0$ by \eqref{def_Dinf}. Therefore, $(G^1)^T x \leq g^1$ is implied by $(G^0)^T x \leq g^0$ and $H^T x = h$. By \cite[Proposition 2]{sherali1995reformulation}, all of the RLT constraints obtained from $(G^1)^T x \leq g^1$ are implied by the RLT constraints obtained from $(G^0)^T x \leq g^0$ and $H^T x = h$. Therefore, we have $(\hat x,\hat X) \in \cF$ if and only if $(G^0)^T \hat x \leq g^0$, $H^T \hat x = h$, $H^T \hat X = h \hat x^T$, and
\[
(G^0)^T \hat X G^0 - (G^0)^T \hat x (g^0)^T - g^0 \hat x^T G^0 + g^0 (g^0)^T \geq 0. 
\]
Note that all of the inequality constraints of $\cF$ are active at $(v,vv^T)$ (or at any $(v^\prime,v^\prime (v^\prime)^T)$, where $v^\prime \in F_0$). If the feasible region of the dual problem given by (RLT-D) is nonempty, then any $(v,vv^T) \in \cF$, where $v \in F_0$, satisfies the optimality conditions of Lemma~\ref{opt_cond} together with any feasible solution of (RLT-D). It follows that any such $(v,vv^T) \in \cF$ is an optimal solution of (RLT). By Proposition~\ref{exactness_cond}, we conclude that the RLT relaxation is exact. On the other hand, if (RLT-D) is infeasible, then (RLT) is unbounded below by linear programming duality. The assertion follows.
\end{proof}

We conclude this section with the following corollary.

\begin{corollary} \label{one_vertex_cor}
Suppose that $F$ given by \eqref{def_F} has exactly one vertex $v \in F$. Then, $\cF$ given by \eqref{def_cF} has exactly one vertex $(v,vv^T)$. Furthermore, if $F = \{v\}$, then $\cF = \{(v,vv^T)\}$.    
\end{corollary}
\begin{proof}
Suppose that $F$ has one vertex $v \in F$. By Proposition~\ref{vertex1}, $(v,vv^T)$ is a vertex of $\cF$. By Lemma~\ref{one_min_face}, Proposition~\ref{exactness_cond}, and Lemma~\ref{vertex_chars}(iv), any vertex of $\cF$ is given by $(v^\prime, v^\prime (v^\prime)^T)$, where $v^\prime$ is a vertex of $F$, which establishes the first assertion. The second assertion follows from Lemma~\ref{bounded_imp} and the first assertion.
\end{proof}

\section{Concluding Remarks} \label{Sec7}

In this paper, we studied various relations between the polyhedral properties of the feasible region of a quadratic program and its RLT relaxation. We presented necessary  and sufficient conditions for the set of instances of quadratic programs that admit exact RLT relaxations. We then discussed how our results can be used to construct quadratic programs with an exact, inexact, and unbounded RLT relaxation.

For RLT relaxations of general quadratic programs, we are able to establish a partial characterization of the set of vertices of the feasible region of the RLT relaxation. We intend to work on a complete characterization of this set in the near future. Such a characterization may have further algorithmic implications for constructing a larger set of instances with inexact but finite RLT relaxations.

Our results in this paper establish several properties of RLT relaxations of quadratic programs. Another interesting question is how the structural properties change for higher-level RLT relaxations.

\bibliographystyle{abbrv}
\bibliography{references}

\end{document}